\documentclass[11pt]{article}
\usepackage[margin=1in]{geometry}
\usepackage{times}
\usepackage{amsmath,amsthm,amssymb,amsfonts}
\usepackage{wrapfig,color}
\usepackage{mathrsfs,upgreek}
\usepackage[colorlinks,linkcolor=black,bookmarksopen,
bookmarksnumbered,citecolor=black,urlcolor=black]{hyperref}
\usepackage{graphicx} 

\newcommand{\R}{ {\mathbb R} }
\newcommand{\Z}{ {\mathbb Z} }

\newcommand{\vb}{ \mathbf{b} }
\newcommand{\vc}{ \mathbf{c} }
\newcommand{\ve}{ \mathbf{e} }
\newcommand{\vf}{ \mathbf{f} }
\newcommand{\vk}{ \mathbf{k} }
\newcommand{\vm}{ \mathbf{m} }
\newcommand{\vp}{ \mathbf{p} }

\newcommand{\vw}{ \mathbf{w} }
\newcommand{\vx}{ \mathbf{x} }
\newcommand{\vy}{ \mathbf{y} }
\newcommand{\vz}{ \mathbf{z} }
\newcommand{\vzero}{ \mathbf{0} }
\newcommand{\Ker}{\operatorname{Ker}}
\newcommand{\rank}{\operatorname{rank}}

\newcommand{\toX}{\arrowvert_{\mathscr{X}}}
\newcommand{\ord}{{\text{th}}}   
\newcommand{\Bdmat}{[\partial_q]}

\DeclareMathOperator{\boundary}{\partial}
\newcommand{\abs}[1]{\lvert#1\rvert}
\newcommand{\homo}{{\sf H}}
\newcommand{\chains}{{\sf C}}
\newcommand{\cycles}{{\sf Z}}
\newcommand{\bdys}{{\sf B}}

\DeclareMathOperator{\image}{{\mathrm im\,}}
\newcommand{\rltvbndry}[3]{\operatorname{\boundary}_{#1}^{\, (#2,#3)}}

\definecolor{darkgrn}{rgb}{0, 0.8, 0}

\newtheorem{theorem}{Theorem}[section]
\newtheorem{lemma}[theorem]{Lemma}
\newtheorem{corollary}[theorem]{Corollary}
\theoremstyle{definition}
\newtheorem{definition}[theorem]{Definition}
\theoremstyle{remark}
\newtheorem{remark}[theorem]{Remark}

\begin{document}

\title{Non Total-Unimodularity Neutralized Simplicial Complexes}

\author{
Bala Krishnamoorthy\thanks{
Department of Mathematics,
Washington State University, Pullman, WA, USA.
Email: {\tt bkrishna@math.wsu.edu}}
\and
Gavin Smith\thanks{
Department of Mathematics,
Washington State University, Pullman, WA, USA.
Email: {\tt gsmith@math.wsu.edu}} 
}

\date{}

\maketitle

\begin{abstract}
  Given a simplicial complex $K$ with weights on its simplices and a
  chain on it, the Optimal Homologous Chain Problem (OHCP) is to find
  a chain with minimal weight that is homologous (over $\Z$) to the
  given chain. The OHCP is NP-complete, but if the boundary matrix of
  $K$ is totally unimodular (TU), it becomes solvable in polynomial
  time when modeled as a linear program (LP). We define a condition on
  the simplicial complex called non total-unimodularity neutralized,
  or {\em NTU neutralized}, which ensures that even when the boundary
  matrix is not TU, the OHCP LP must contain an integral optimal
  vertex for every input chain. This condition is a property of $K$,
  and is independent of the input chain and the weights on the
  simplices. This condition is strictly weaker than the boundary
  matrix being TU. More interestingly, the polytope of the OHCP LP may
  not be integral under this condition. Still, an integral optimal
  vertex exists for every right-hand side, i.e., for every input
  chain. Hence a much larger class of OHCP instances can be solved in
  polynomial time than previously considered possible. As a special
  case, we show that $2$-complexes with trivial first homology group
  are guaranteed to be NTU neutralized.
\end{abstract}

\section{Introduction} \label{sec-Intro}

Topological cycles in shapes capture their important features, and are
employed in many applications from science and engineering. A problem
of particular interest in this context is the optimal homologous cycle
problem, OHCP, where given a cycle in the shape, the goal is to
compute the shortest cycle in its topological class (homologous).  For
instance, one could generate a set of cycles from a simplicial complex
using the persistence algorithm \cite{EdLeZo2002} and then tighten
them while staying in their respective homology classes. The OHCP and
related problems have been widely studied in recent years both for two
dimensional complexes
\cite{ChVeErLaWh2008,ChErNa2009,ChFr2010,VeEr2006,DeLiSuCo2008} and
for higher dimensional instances \cite{SiGh2007, TaJa2009}. The OHCP
with homology defined over the popularly used field of $\Z_2$ was
known to be NP-hard \cite{ChFr2010a}. But it was shown recently that
if the homology is defined over $\Z$, then one could solve OHCP in
polynomial time when the simplicial complex has no relative torsion
\cite{DeHiKr2011}. The generalized decision version of the problem
considering chains instead of cycles (also termed OHCP) was recently
shown to be NP-complete \cite{DuHi2011}. Instances that fall in
between these two extreme cases have not been studied so far. In
particular, the complexity of OHCP in the presence of relative torsion
is not known.

The polynomial time solvability of OHCP was shown by modeling the
problem as a linear program (LP), and showing that the constraint
matrix of this LP is totally unimodular (TU) when the simplicial
complex does not have any relative torsion \cite{DeHiKr2011}. This
connection between TU matrices and polynomial time solvability of
integer programs (IPs) by solving their associated LPs is well known,
e.g., see \cite[Chap.~19--21]{Schrijver1986}. If the constraint matrix
of an LP is TU, then its polyhedron is integral, i.e., all its
vertices have integral coordinates. Two other concepts associated with
integral polyhedra that are weaker than TU matrices are $k$-balanced
matrices \cite{CoCoTr1994,CoCoVu2006} and totally dual integral (TDI)
systems \cite[Chap.~22]{Schrijver1986} \cite{DiFeZa2008}. But
applications of such weaker conditions to the OHCP LP and their
potential correspondences to the topology of simplicial complexes have
not been explored so far.

\paragraph{Our Contributions:}
We define a characterization of the simplicial complex termed non
total-unimodularity neutralized, or NTU neutralized for short, which
guarantees that even when there is relative torsion in the simplicial
complex, every instance of the OHCP LP has an integer optimal
solution. Under this condition, the OHCP instance for any input chain
with homology defined over $\Z$ could be solved in polynomial time
using linear programming even when the constraint matrix of the OHCP
LP is not TU. We arrive at our main result by studying the structure
of the OHCP LP, and characterizing several properties of its basic
solutions. Recall that the vertices of an LP correspond to its basic
feasible solutions. In particular, we prove that an OHCP LP for a
given input chain has a fractional basic solution if and only if the
OHCP LP with a component {\em elementary} chain, i.e., a chain with a
single nonzero coefficient of $1$, as input has a certain fractional
basic solution. Using this result, we show that no OHCP LP over the
given simplicial complex has a unique fractional optimal solution if
and only if every elementary chain involved in each relative torsion
has a {\em neutralizing chain} in the complex, i.e., when the
simplicial complex is NTU neutralized.

Our result partly fills the gap between the extreme cases of the OHCP
LP solving the OHCP when the complex has no relative torsion, and the
OHCP being NP-complete (there could well be instances that are
amenable to efficient solutions by methods distinct from solving
LPs). The condition of a complex being NTU neutralized is strictly
weaker than requiring its boundary matrix to be TU, or even to be
balanced. Further, this condition is a property of the simplicial
complex, and is independent of the input chain as well as the choice
of weights on the simplices. Hence a much broader class of OHCP
instances can be solved in polynomial time using linear programming
than previously considered possible. When the simplicial complex is
NTU neutralized, the linear system in the {\em dual} of the OHCP LP is
TDI. But this case does not appear to be covered by any of the
currently known characterizations of TDI systems. In particular, the
polytope of the OHCP LP may {\em not} be integral even when the
complex is NTU neutralized. Still, an integral optimal solution exists
for {\em every} integral right-hand side, i.e., for {\em every} input
chain. As a special case, we show that every $2$-complex with trivial
first homology group is guaranteed to be NTU neutralized.

\subsection{An example, and some intuition} \label{ssec-exmpl}
We illustrate the condition of a simplicial complex being NTU
neutralized by describing a set of two dimensional complexes related
to the M\"obius strip. A $2$-complex having no relative torsion is
equivalent to it having {\em no M\"obius strip}
\cite[Thm.~5.13]{DeHiKr2011}. Consider the three different
triangulations of a space in Figure~\ref{fig-threecplxs}. In the left
and right complexes, we have a M\"obius strip self-intersecting at one
($d$) and two vertices ($a,d$), respectively, resulting in relative
torsion in both cases. In the middle complex, the self intersection is
along the {\em edge} $ad$, hence we do not have relative
torsion. Hence the boundary matrix is TU only for the middle complex.
Still, in the right complex, the OHCP LP has an integral optimal
solution for every input chain.
\begin{figure}[ht!]
  \centering 
  \includegraphics[width=\columnwidth]{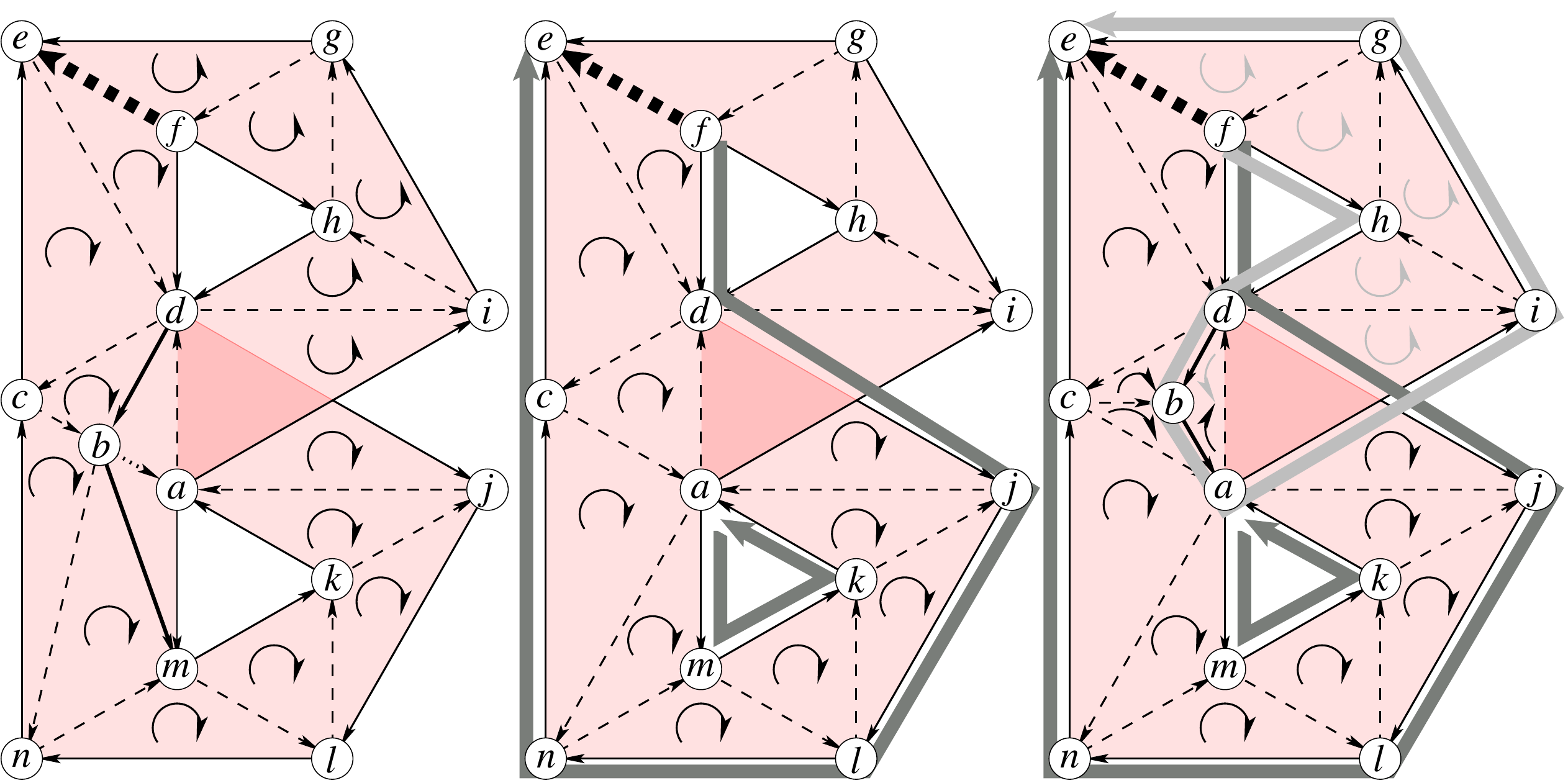} 
  \caption{Three triangulations of a space. The right complex is NTU
    neutralized, the left one is not. The middle complex has a TU
    boundary matrix.}
  \label{fig-threecplxs}
\end{figure}

For example, consider the edge $ef$ (shown in thick dashes) with
multiplier $1$ as the input chain. Let the edge weights be as follows:
dashed and dotted edges have weight $1$, thin solid edges have weights
of $0.05$, and thick solid edges have weights of $0.10$ each. The
dashed edges are the ``manifold'' edges in the potential M\"obius
strip in each complex. The thin solid edges are boundary edges. The
two pairs of thick solid edges are boundary edges in the candidate
M\"obius strips, but are each shared by two triangles in the
simplicial complex. Solving the OHCP involves pushing the heavy
manifold edge(s) onto the light boundary edges using the boundaries of
triangles. In the left complex, the unique optimal solution to the
OHCP LP corresponds to all the black solid edges, which is the
boundary of the M\"obius strip self-intersecting at vertex $d$, with
coefficients $\pm0.5$.  This instance illustrates the case of minimal
violation of TU we study in a general simplicial complex.

The optimal homologous chain is indicated in dark gray in the middle
and right complexes (it is the same in all three complexes).  In the
right complex, there are two integral optimal solutions to the OHCP
LP, which are outlined in dark and light gray. Any convex combination
of these two chains also corresponds to an optimal solution of the
OHCP LP, including the one made of all solid edges with coefficients
$\pm0.5$. This observation may be explained by the presence of a disk
whose boundary is an odd number of dashed edges, e.g., triangle $adc$,
which {\em neutralizes} the M\"obius strip. This ``odd disk''
provides an alternative to pushing the heavy manifold edge onto all
the light boundary edges by going around the entire M\"obius strip
with fractional multipliers. Instead, one could take a ``shortcut''
across the middle of the strip through the neutralizing chain,
permitting integer multipliers. In this case, there also exists a
complementary shortcut. The fractional solution going all the way
around the strip is a convex combination of the two shortcut integral
solutions. Our characterization of NTU neutralization generalizes this
observation to arbitrary dimensions. Intuitively, a complex is NTU
neutralized if there exists an ``odd disk'' providing such a shortcut
across every relative torsion, for each of its ``manifold'' elementary
chains (similar to edge $ef$ above).

Adding triangle $akm$ to the left complex makes it NTU
neutralized. For instance, $adcbnmlkja$ is a disc whose boundary is
$9$ dashed edges. Alternatively, adding both triangles $akm$ and $dfh$
to the left complex also makes it NTU neutralized. In this case, the
first homology group becomes trivial, which is a sufficient condition
for $2$-complexes to be NTU neutralized (Theorem
\ref{thm-1_hom_triv_NTUNeut}).

\section{Background} \label{sec-backg}

We recall some relevant basic concepts and definitions from algebraic
topology and optimization. Refer to standard books, e.g., ones by
Munkres~\cite{Munkres1984} and by Schrijver \cite{Schrijver1986}, for
details.

Given a vertex set $V$, a \emph{simplicial complex} $K=K(V)$ is a
collection of subsets $\{\sigma \subseteq V\}$ where $\sigma'\subseteq
\sigma$ is in $K$ if $\sigma\in K$.  A subset $\sigma\in K$ of
cardinality $q=p+1$ is called a \emph{$p$-simplex}.  If
$\sigma'\subseteq \sigma$ ($\sigma'\subset \sigma$), we call $\sigma'$
a \emph{face} (\emph{proper face}) of $\sigma$, and $\sigma$ a
\emph{coface} (\emph{proper coface}) of $\sigma'$.  An oriented
simplex $\sigma = \{v_0, v_1, \cdots, v_p\}$ or $v_0 v_1 \cdots v_p$
is an ordered set of vertices.  The simplices $\sigma_i$ with
coefficients $\alpha_i$ in $\Z$ can be added formally creating a chain
$c = \Sigma_i \alpha_i \sigma_i$.  These chains form the chain group
$\chains_p$.  The boundary $\partial_p \sigma$ of a $p$-simplex
$\sigma$, $p\geq 0$, is the $(p-1)$-chain that adds all the
$(p-1)$-faces of $\sigma$ considering their orientations.  This
defines a boundary homomorphism $\partial_p: \chains_p\rightarrow
\chains_{p-1}$.  The kernel of $\partial_p$ forms the $p$-cycle group
$\cycles_p(K)$ and its image forms the $(p-1)$-boundary group
$\bdys_{p-1}(K)$.  The homology group $\homo_p(K)$ is the quotient
group $\cycles_p(K)/\bdys_p(K)$.  Intuitively, a $p$-cycle is a
collection of oriented $p$-simplices whose boundary is zero. It is a
nontrivial cycle in $\homo_p$, if it is not a boundary of a
$q$-chain. 

For a finite simplicial complex $K$, the groups of chains
$\chains_p(K)$, cycles $\cycles_p(K)$, and $\homo_p(K)$ are all
finitely generated abelian groups. By the fundamental theorem of
finitely generated abelian groups \cite[page 24]{Munkres1984} any such
group $G$ can be written as a direct sum of two groups $G=F \oplus T$
where $F\cong (\Z \oplus\cdots\oplus \Z)$ and
$T\cong(\Z/t_1\oplus\cdots\oplus \Z/t_k)$ with $t_i>1$ and $t_i$
dividing $t_{i+1}$. The subgroup $T$ is called the \emph{torsion} of
$G$. If $T=0$, we say $G$ is \emph{torsion-free}. 

For a subcomplex $L_0$ of a simplicial complex $L$, the quotient group
$\chains_p(L)/\chains_p(L_0)$ is called the group of \emph{relative
  $p$-chains} of $L$ modulo $L_0$, denoted $\chains_p(L,L_0)$.  The
boundary operator $\boundary_p \colon \chains_p(L)\rightarrow
\chains_{p-1}(L)$ and its restriction to $L_0$ induce a
homomorphism
\[ \rltvbndry{p}{L}{L_0} \colon \chains_p(L,L_0)
\rightarrow \chains_{p-1}(L,L_0)\,.\]
Writing $\cycles_p(L,L_0)={\rm ker}\rltvbndry{p}{L}{L_0}$ for
\emph{relative cycles} and $\bdys_{p-1}(L,L_0)=\image
\rltvbndry{p}{L}{L_0}$ for \emph{relative boundaries}, we obtain the
\emph{relative homology group} 
\[ \homo_p(L,L_0) = \cycles_p(L,L_0)/\bdys_p(L,L_0).\]

Given the oriented simplicial complex $K$ of dimension $d$, and a
natural number $p$, $1 \le p \le d$, the \emph{$p$-boundary matrix} of
$K$, denoted $[\partial_{p}]$, is a matrix containing exactly one
column $j$ for each $p$-simplex $\sigma$ in $K$, and exactly one row
$i$ for each $(p-1)$-simplex $\tau$ in $K$.  If $\tau$ is not a face
of $\sigma$, then the entry in row $i$ and column $j$ is 0.  If $\tau$
is a face of $\sigma$, then this entry is $1$ if the orientation of
$\tau$ agrees with the orientation induced by $\sigma$ on $\tau$, and
$-1$ otherwise.

A matrix $A$ is totally unimodular (TU) if the determinant of each of
its square submatrix is either $0,1$, or $-1$.  Hence each $A_{ij} \in
\{0, \pm 1\}$ as well.  The importance of TU matrices for integer
programming is well known \cite[Chapters 19-21]{Schrijver1986}. In
particular, it is known that the {\em integer} linear program
\begin{equation} \label{eqn-IP}
  \min\,\{ \vf^T \vx ~|~ A \vx = \vb, \; \vx \ge \vzero, \vx \in
  \Z^n\}
\end{equation}
for $A \in \Z^{m \times n}, \vb \in \Z^n$ can {\em always}, i.e., for
every $\vf \in \R^n$, be solved in polynomial time by solving its
linear programming {\em relaxation} (obtained by ignoring $\vx \in
\Z^n$) if and only if $A$ is totally unimodular.  This result was
employed to show that the OHCP for the input $p$-chain $\vc$ modeled
as the following LP could be solved to get integer solutions under
certain conditions \cite[Eqn.~(4)]{DeHiKr2011}.
\begin{align}  \label{eqn-origOHCPLP} 
  \min \;&  \sum_i \, \abs{w_i} \, (x_i^+ + x_i^-)\notag\\
  \text{subject to}\quad &
  \vx^+ - \vx^- = \vc + [\boundary_{q}]\; \vy\\
  & \vx^+, \; \vx^- \ge \mathbf{0} \, . \notag
\end{align}
We assume the weights $w_i$ for $p$-simplices are
nonnegative. Replacing $\vy$ with two nonnegative variable vectors
$\vy^+$ and $\vy^-$, we rewrite the above LP in the following form.
\begin{align}
\min~ & 
\begin{bmatrix}
\vw^T & \vw^T & \vzero^T & \vzero^T
\end{bmatrix}
\vz  \nonumber \\
\label{eqn-OHCP}
\text{subject to}~~ & 
\begin{bmatrix}
~I & \,-I & \,-B & ~B~
\end{bmatrix}
\vz = \vc
\\
& \mbox{\hspace*{1.395in}}\vz \ge \vzero. \nonumber
\end{align}
Notice that $B=\Bdmat$, and the variable vector $\displaystyle \vz^T
= \begin{bmatrix} \vx^{+T} & \vx^{-T} & \vy^{+T} &
  \vy^{-T} \end{bmatrix}$. Recall that $x^+_i$ and $x^-_i$ correspond
to the $i\ord$ $p$-simplex, while $y^+_j, y^-_j$ capture the
coefficients for the $j\ord$ $q$-simplex. We refer to this formulation
as the OHCP LP from now on.  We let $P$ denote its feasible region,
and let $A = \begin{bmatrix} I & -I & -B & B \end{bmatrix}$ be the
constraint matrix of (\ref{eqn-OHCP}).  It was shown that $A$ is TU,
or equivalently, $P$ is integral if and only if $B$ is TU, which
happens \cite[Thm.~5.2]{DeHiKr2011} if and only if
$\homo_p\left(L,L_0\right)$ is torsion-free, for all pure subcomplexes
$L_0$, $L$ in $K$ of dimensions $p$ and $q$ respectively, where $L_0
\subset L$. Thus OHCP can be solved in polynomial time if the
simplicial complex is free of relative torsion.

The point $\vz$ is a {\em vertex} of $P$ if it is in $P$, but is not a
convex combination of any two distinct elements of $P$
\cite[Chap.~8]{Schrijver1986}. A \emph{basic solution} of a system of
linear equations is a point in a solution space of dimension $d$ where
a set of $d$ linearly independent constraints are active, i.e.,
satisfied as equations. If a basic solution of $P$ is feasible, then
it is a vertex \cite[Chap.~8]{Schrijver1986}.

\section{Characterizations of Basic Solutions of the OHCP LP} \label{sec-Basic_Sols_Char}

Our goal is to characterize the fractional basic feasible solutions,
or vertices, of the OHCP LP. Instead, we establish several properties
of basic solutions, by relaxing feasibility. This step simplifies the
analysis, and we prove that the basic solutions and vertices are
equivalent in a certain sense as explained below (see
Corollary~\ref{cor-equiv_vert}).

Notice that $P$ is the hyperplane defined by the equality constraints,
with the only bounds being the nonnegativity constraints.  We use
$P_A$ to denote the hyperplane that is $P$ without the bounds.  We use
$\vz$ to refer to a general element of $\R^{2(m+n)}$, and call $z_i$
an $x$-entry if $i \leq 2m$, and a $y$-entry if $i > 2m$.
\begin{definition}
  \label{def-concise}
  For any entry $z_i$ of $\vz \in \R^{2(m+n)}$, its \emph{opposite
    entry} is $z_{i+m}$ for $i \leq m, z_{i-m}$ for $m < i \leq 2m,
  z_{i+n}$ for $2m < i \leq 2m+n$, and $z_{i-n}$ for $2m+n < i \leq
  2(m+n)$.  We denote the opposite entry of $z_i$ as $z_{-i}$.  Any
  pair of opposite entries are coefficients for the same simplex.
  Hence for a pair of opposite entries $z_i, z_{-i}$ of $\vz$, if at
  least one of the two is 0, then $\vz$ is \emph{concise} in the
  $i\ord$ entry.  $\vz$ is \emph{concise} if it is concise in each
  entry.
\end{definition} 
The following definition translates coordinates of an OHCP LP solution
to the $i\ord$ row or $j\ord$ column of $\Bdmat$, and the $p$- and
$q$-simplices the row and column represent, respectively.
\begin{definition}
  \label{def-p-coeffs}
  For a solution $\vz$ of an OHCP LP, for any $i \leq m$, the
  \emph{$i\ord\ p$ coefficient} is $z_i - z_{-i}$, and for any $j: 2m
  < j \leq 2m + n$, the \emph{$(j - 2m)\ord\ q$ coefficient} is $z_j -
  z_{-j}$.
\end{definition}

In figures of simplices representing solutions to OHCP LPs, we
generally show the $p$- and $q$-coefficients of simplices, and assume
that all solutions illustrated are concise.  When we call a set of
solutions \emph{equivalent} we mean each has the same $p$- and
$q$-coefficients.

For any OHCP LP, there is the unique feasible concise solution where
all the $y$-coordinates are $0$.  We call this solution the
\emph{identity} solution, and denote it $\vz^I$.  For a given
simplicial complex $K$ and the constraint matrix $A$ associated with
its OHCP LP instances, we use $\vz^K$ to refer to an element of
$\Ker(A)$, the kernel of $A$.  For any integral $\vz^K$, the set of
$p$-coefficients of $\vz^K$ represent a $p$-chain that is
null-homologous in $K$.  We list some rather straightforward results
from linear algebra.
\begin{enumerate} 
\item Any $\vz \in P_A$ may be written as $\vz^I + \vz^K$.
\item Given $\vz \in P_A,\vz = \vz^0 + \vz^K$, then $\vz^0 \in P_A$ if
  and only if $\vz^K \in \Ker(A)$.
\item Because $A$ is rational, for any $\vz^K \in \Ker(A)$, there is
  some scalar $\alpha > 0$ such that $\alpha\vz^K$ is integral.
\end{enumerate}
The following theorem is foundational to many of our later results.
\begin{theorem}
  \label{thm-vk_active}
  Let $\vz \in P_A$.  $\vz$ is a basic solution if and only if
  $~\forall \vz^K \in \left(\Ker(A) \setminus
  \{\vzero\}\right),\ \exists i: z_i = 0, z_i^K \neq 0.$
\end{theorem}
\begin{proof}
  We prove both directions by contrapositive.  Assume for some $\vz^K
  \neq \vzero$, there is no such $i$. Because $\vz^K \neq \vzero, \vz
  + \vz^K \neq \vz - \vz^K$. If $\vz \in P_A$ and $\vz^K \in \Ker(A),
  \vz + \alpha\vz^K \in P_A\ \forall \alpha \in \R$.  Therefore the
  line segment $L \subset \R^{2(m+n)}$ defined by the two distinct end
  points $\vz \pm \vz^K$ is contained in $P_A$. Hence all the equality
  constraints are active at all points in $L$.
  Consider an arbitrary inequality constraint $j := z_j \geq 0$.  If
  $j$ is active at $\vz$, then since there is no such $i$, $z_j^K =
  0$, and so $j$ must be active for all points in $L$.  Therefore any
  constraint active at $\vz$ is active at any point in $L$.  Therefore
  $\vz$ cannot be a basic solution.

  Now assume $\vz$ is not a basic solution.  Therefore there exists
  some line segment $L$ with $\vz$ in its interior where all
  constraints active at $\vz$ are active at all points in $L$.  Since
  $\vz \in P_A$, all equality constraints are active in $L$, so $L
  \subset P_A$.  Therefore, for any other interior point $\vz^0$ of
  $L$ we have $\vz^0 \neq \vz$ and hence $\vz - \vz^0 = \vz^K \in
  \left(\Ker(A) \setminus \{\vzero\}\right)$. Since all inequality
  constraints active at $\vz$ are active at $\vz^0$, we have $z_i = 0
  \implies z_i^0 = 0$, and hence $z_i^K = 0$.
\end{proof}

\begin{figure}[ht!]
  \centering
  \includegraphics[scale=0.8]{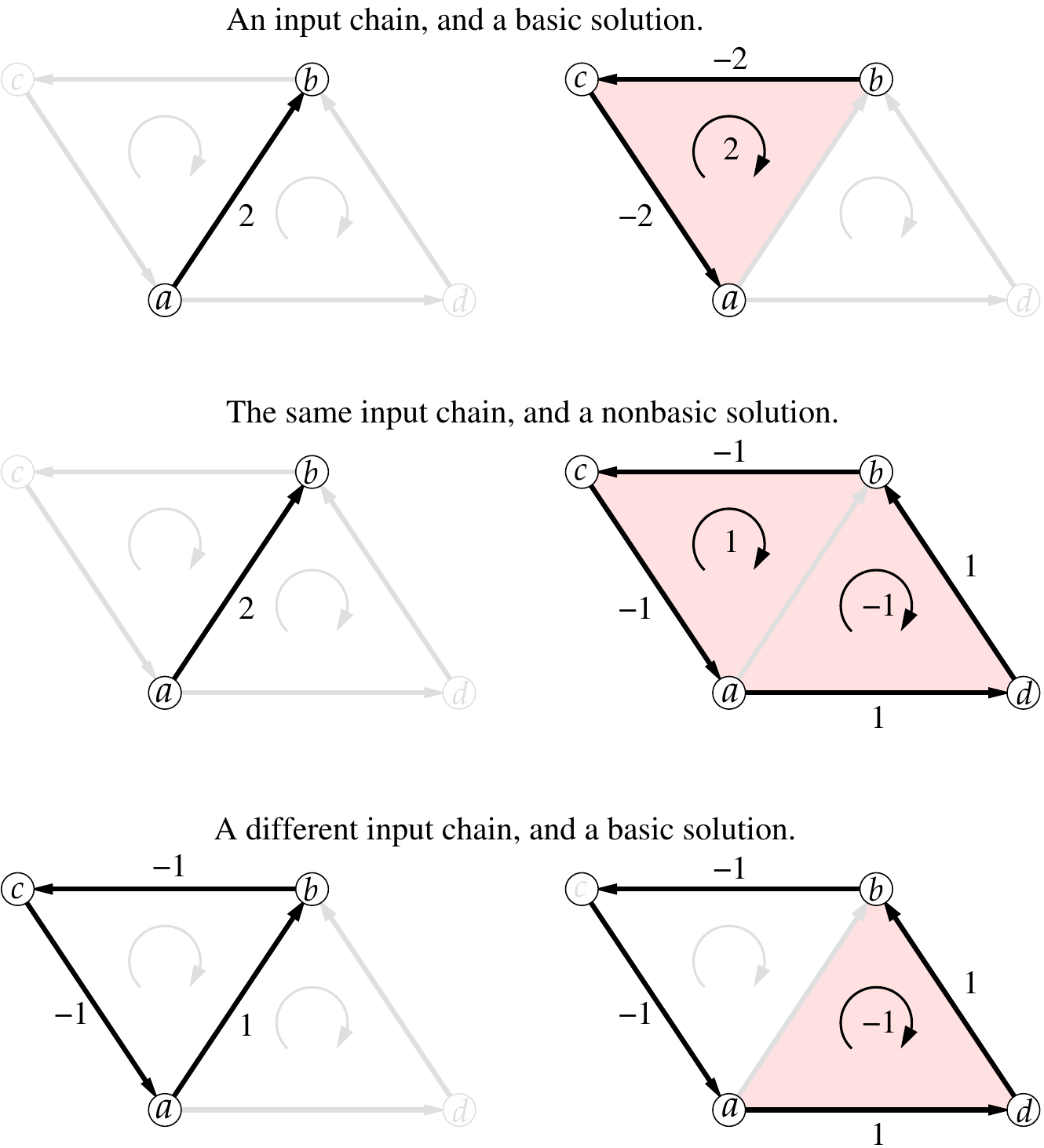}
  \caption{Simple examples of basic and nonbasic solutions.}
  \label{fig-Basic_Sol_Ex_1}
\end{figure} 

Figure~\ref{fig-Basic_Sol_Ex_1} illustrates nonbasic and basic
solutions of the OHCP LP in a $2$-complex.  Orientations of simplices
in $K$, coefficients of the input chains, and the $p$- and
$q$-coefficients of the solutions are shown.  Note that the
$p$-coefficients are the same for the second nonbasic solution and the
last basic solution. Whether or not a solution is basic can depend on
the $q$-coefficients and the input chain.

\bigskip
\noindent Consider the $2(m+n) \times (m + 2n)$ matrix
$ \displaystyle
N = \begin{bmatrix}
I_m &  &B \\
I_m \\
& I_n & I_n\\
& I_n
\end{bmatrix}
$ (entries not specified are zero).
The columns of $N$ form a basis of $\Ker(A)$.  Analyzing its
structure, together with Theorem~\ref{thm-vk_active}, yields the
following results, whose proofs we omit.
\begin{lemma}
  \label{lem-basic_sol_concise}
  Any basic solution of an OHCP LP is concise.
\end{lemma}
\begin{lemma}
  \label{lem-vk_equiv_bi}
  Any $\vz^K \in \Ker(A)$ is equivalent to a linear combination of the
  last $n$ columns of $N$.
\end{lemma}
\begin{corollary}
  \label{cor-vk_conc_then_y_nz}
  If $\vz^K \in \left(\Ker(A) \setminus \{\vzero\}\right)$ is concise,
  then at least one $y$-coordinate in $\vz^K$ is nonzero.
\end{corollary}

\begin{lemma}
  \label{lem-last_OHCP_term_cancels}
  Let $\vz\in P_A$ be a basic solution. Let $\vz^0 \in P_A$ with
  $\vz^0$ concise in all $x$-entries.  Let $\vz = \vz^0 + \vz^K$, with
  $\vz^K$ being concise in all $x$-entries, and for each
  $y$-coordinate $j, z_j^0 \neq 0 \implies z_j \neq 0$. Then $\vz^K
  \neq \vzero$ if and only if there exists a $p$-coefficient that is 0
  in $\vz$, but nonzero in $\vz^0$.
\end{lemma}
\begin{proof}
  If there is a $p$-coefficient that is 0 in $\vz$, but nonzero in
  $\vz^0$, then $\vz^K \neq \vzero$ simply because $\vz = \vz^0 +
  \vz^K$.  Now assume $\vz^K \neq \vzero$.  Then $\vz \neq \vz^0$.
  Because $\vz$ is basic, it is the only point in $P_A$ where all
  entries that are 0 at $\vz$ are 0.  Since for each $y$-coordinate
  $j, z_j^0 \neq 0 \implies z_j \neq 0$, there must be some $x$
  entries that are zero at $\vz$, but nonzero at $\vz^0$.  Let $i$ be
  one such entry.  Since $i$ is 0 at $\vz$ but nonzero at $\vz^0$, it
  must be nonzero in $\vz^K$.  Since $\vz^0$ and $\vz^K$ are concise
  in all $x$-entries, $-i$ must be 0 in both $\vz^0$ and $\vz^K$, and
  therefore also at $\vz$.  Therefore the $p$-coefficient
  corresponding to $i$ and $-i$ must be 0 in $\vz$ but nonzero in
  $\vz^0$.
\end{proof}

For $p = 1$, Lemma \ref{lem-last_OHCP_term_cancels} is saying that if
we attempt to get to a basic solution by adding a set of triangles to
the input chain, then adding that set of triangles must completely
cancel at least one edge.  Referring to the two basic solutions shown
in Figure \ref{fig-Basic_Sol_Ex_1}, the edge of the input chain
canceled is edge $ab$ in both cases.

\begin{definition}
  \label{def-coll_concise}
  A set of vectors is \emph{linearly concise} if any linear
  combination of the set is concise.
\end{definition}

The next result establishes a method for decomposing a nonbasic
solution into a basic solution and a remainder element of $\Ker(A)$,
which we will use in later analysis.  
\begin{theorem}
  \label{thm-v_plus_vk_basic}
  Let $\vz^0 \in P_A$ be a basic solution.  Let $\vz^K \in \Ker(A)$
  with $\{\vz^0, \vz^K\}$ linearly concise.  Let $\vz = \vz^0 +
  \vz^K$.  Then $\vz$ is a basic solution if and only if there do not
  exist $\vz^C, \vz^D$ satisfying the following properties:
  \begin{enumerate}
  \item \label{cnd-sum_vk} $\vz^C + \vz^D = \vz^K$.
  \item \label{cnd-in_KerA} $\vz^C, \vz^D \in \Ker(A)$.
  \item \label{cnd-vr_nz} $\vz^D \neq \vzero$.
  \item \label{cnd-coll_conc} $\{\vz^0, \vz^K, \vz^D\}$ is linearly
    concise.
  \item \label{cnd-v'basic} $\vz^0 + \vz^C = \vz^1$ is a basic
    solution.
  \item \label{cnd-no_y_elim} For each $y$-coordinate $j, z_j^1 \neq 0
    \implies z_j \neq 0$.
  \item \label{cnd-v_contains_vr} For each $x$-coordinate $i, z_i^D
    \neq 0 \implies z_i \neq 0$.
  \end{enumerate}
\end{theorem}

\begin{proof}
  Suppose there exists such a decomposition of $\vz^K$.  By
  Property~\ref{cnd-in_KerA}, and that $\vz^0 \in P_A$, we have that
  $\vz \in P_A$.  Therefore if $\vz$ is a basic solution, then the
  vectors $\vz, \vz^1, \vz^D$ satisfy the conditions for $\vz, \vz^0,
  \vz^K$ in Lemma~\ref{lem-last_OHCP_term_cancels}.  Because of
  Property~\ref{cnd-coll_conc}, $\{\vz^0, \vz^K, \vz^C, \vz^D,
  \vz^1,\vz\}$ is linearly concise.  Therefore $\left(\forall
  x\text{-coordinate } i \right.$, $\left. z_i = z_{-i} \implies z_i^1
  = z_{-i}^1\right)$ $\iff$ $\left(\forall x\text{-coordinate } i
  \right.$, $z_i = 0$ $\left.\implies z_i^1 = 0\right)$ $\iff$
  $\left(\forall x\text{-coordinate } i \right.$, $z_i = 0$ $\left.
  \implies z_i^D = 0\right)$ $\iff$ $\left(\forall x\text{-coordinate
  } i\right.$, $z_i^D \neq 0 \implies$ $\left.z_i \neq 0\right)$.  So
  by Lemma~\ref{lem-last_OHCP_term_cancels}, $\vz$ is not a basic
  solution.

  Now suppose $\vz$ is not a basic solution. $\{\vz^0, \vz^K, \vz\}$
  is still linearly concise.  Construct $\vz^D$ and find $\vz^1$ using
  the following algorithm.
  \begin{enumerate}
  \item Let $\vz^D = \vzero, \vz^1 = \vz$.  Then $\{\vz^0, \vz^K,
    \vz^D, \vz^1\}$ is linearly concise.
  \item 
    \label{stp-vprime_not_basic}
    $\vz^1$ must be in $P_A$, and is not a basic solution.  By
    Theorem~\ref{thm-vk_active}, $\exists \, \vz^N \in (\Ker(A)
    \setminus \{\vzero\})$ where $z_i^N \neq 0 \implies z_i^1 \neq 0$.
    Because $\{\vz^0, \vz^K, \vz^D, \vz^1\}$ is linearly concise,
    $\{\vz^0, \vz^K, \vz^D, \vz^1,\vz^N\}$ is linearly concise.
  \item By Corollary~\ref{cor-vk_conc_then_y_nz}, $z_i^N \neq 0$ for
    some $y$-coordinate $i$.  Find $i$ such that $z_j^N \neq 0, j > 2m
    \implies \abs{z_i^1 / z_i^N} \leq \abs{z_j^1 / z_j^N}$.
  \item Let $\alpha = z_i^1 / z_i^N$.
  \item Let $\vz^D = \vz^D + \alpha\vz^N, \vz^1 = \vz^1 -
    \alpha\vz^N$.  Because we may add any linear combination of a set
    of linearly concise vectors to the linearly concise set, $\{\vz^0,
    \vz^K, \vz^D, \vz^1,\vz^N\}$ is still linearly concise.
  \item IF $\vz^1$ is not a basic solution THEN LOOP to
    Step~\ref{stp-vprime_not_basic}.
  \item STOP.
  \end{enumerate}
  Because $K$ is finite, and we make at least one $y$-entry zero that
  was nonzero in $\vz^1$ in each loop, and do not make any zero
  $y$-entries in $\vz^1$ nonzero. Hence by
  Theorem~\ref{thm-vk_active}, this algorithm must eventually
  terminate. More precisely, it must terminate after at most $n$
  iterations.  By our criteria of choosing $\vz^N$ in each loop,
  $\vz^D, \vz^1$, and $\vz^C = \vz^K - \vz^D$ satisfy all the criteria
  of the theorem.
\end{proof}

\begin{figure}[ht!]
  \centering
  \includegraphics[scale=0.8]{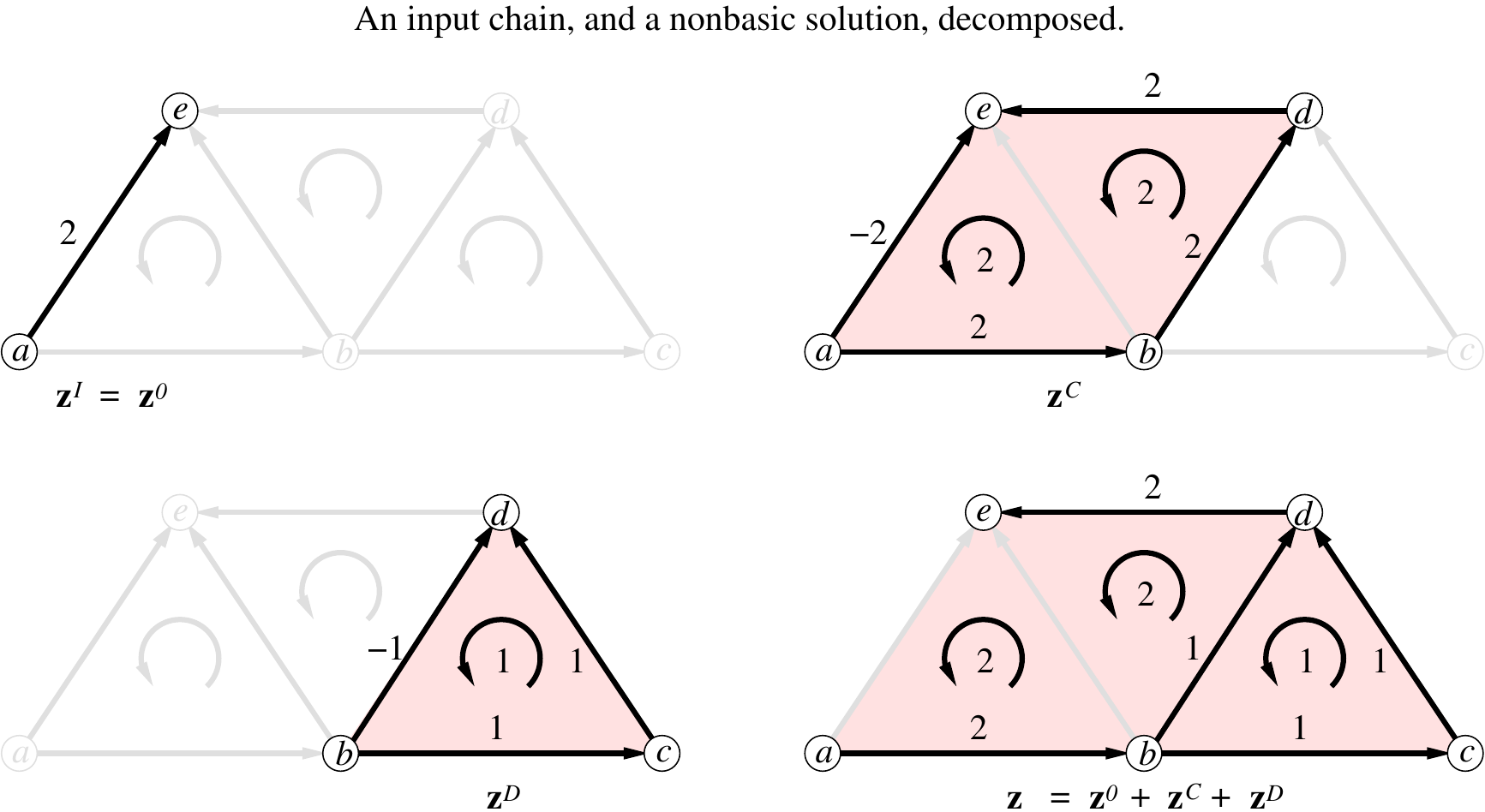}
  \caption{An illustration of $\vz^C$ and $\vz^D$.}
  \label{fig-zC_and_zD}
\end{figure} 

The decomposition described in Theorem~\ref{thm-v_plus_vk_basic}
isolates the portion(s) of a solution that makes it nonbasic.
Figure~\ref{fig-zC_and_zD} illustrates a simple example.  At upper
left is the input chain, equivalent to $\vz^I$, and which takes the
role of $\vz^0$ in the theorem.  Note that the illustration only shows
equivalence classes under $p$- and $q$-coefficients.  The next
Lemma~\ref{lem-must_cancel_two} describes necessary conditions to
transform a nonbasic solution to a basic one.
\begin{lemma}
  \label{lem-must_cancel_two}
  Let $\vz^0 \in P_A$ be concise.  Let $\vz = \vz^0 + \vz^1$ where
  $z_j^0 \neq 0 \implies z_j^1 = 0\ \forall j > 2m$. If $\vz$ is a
  basic solution in $P_A$, then for each $\vz^K \in \left(\Ker(A)
  \setminus \{\alpha\vz^1\}\right)\left(\alpha \in \R\right)$ where
  $z_i^K \neq 0 \implies z_i^0 \neq 0$, there must be two
  $x$-coordinates $r$ and $s$ where $z_r = z_s = 0$, $z_r^K, z_s^K
  \neq 0$, and $\displaystyle \frac{z_r^K}{z_r^1} \neq
  \frac{z_s^K}{z_s^1}$.  Furthermore, if $O_r$ and $O_s$ are the OHCP
  LPs with input chains where the only nonzero coefficients are $r$
  and $s$, respectively, with these coefficients equaling those in
  $\vz^0$, then $\vz^1 + \vz^{I\prime}$ is a basic solution to $O_r$
  or $O_s$ where $\vz^{I\prime}$ is the solution with $\{\vz^1,
  \vz^{I\prime}\}$ linearly concise and equivalent to the identity
  solution for $O_r$ or $O_s$, respectively.
\end{lemma}

\begin{proof}
  By Theorem~\ref{thm-vk_active}, the existence of $\vz^K \in
  \left(\Ker(A) \setminus \{\alpha\vz^1\}\right)$ where $z_i^K \neq 0
  \implies z_i^0 \neq 0$, implies $\vz^0$ is not a basic solution.
  Assume there is such a $\vz^K$ with no such $r$ and $s$.  By $z_j^0
  \neq 0 \implies z_j^1 = 0\ \forall j > 2m$, we get $z_j^K \neq 0
  \implies z_j \neq 0\ \forall j > 2m$. So if there is no
  $x$-coordinate $r$ where $\vz_r^K \neq 0, \vz_r = 0$, then $\vz$
  cannot be a basic solution. Let $\mathscr{R}$ be the set of
  $x$-coordinates where $\vz_r^K \neq 0, \vz_r = 0$, and assume
  $\mathscr{R}$ is nonempty.  Because there is no such $r$ and $s$,
  there is some $\alpha \in \R$ such that $(z_r^K/z_r^1) =
  \alpha\ \forall r \in \mathscr{R}$. If $\vz \in P_A$, then $\vz^1
  \in \Ker(A)$. So then $\vz^K - \alpha\vz^1$ is in $\Ker(A)$, and
  $z_i = 0 \implies z_i^K - \alpha z_i^1 = 0\ \forall i \leq 2m$.
  Because $\vz^K \neq \alpha\vz^1,\ \vz^K - \alpha\vz^1 \neq \vzero$.
  Therefore by Theorem~\ref{thm-vk_active}, $\vz$ is not a basic
  solution.

  Let $\vz^{I\prime}$ be the solution with $\{\vz^1, \vz^{I\prime}\}$
  linearly concise and equivalent to the identity solution for $O_r$.
  If $\vz^{I\prime} + \vz^1$ is not a basic solution to $O_r$, then we
  may decompose $\vz^1$ into $\vz^C + \vz^D$ according to
  Theorem~\ref{thm-v_plus_vk_basic}.  If $\vz^D$ does not bring
  $z_s^0$ to zero in our original OHCP LP $O$, then $\vz$ cannot be a
  basic solution of $O$ because all nonzero coefficients of $\vz^D$
  will be nonzero in $\vz$.  Even if it does bring $z_s^0$ to zero,
  $\vz^0 + \vz^C$ cannot be a basic solution because of the first part
  of this lemma, and because $\vz^D$ does not bring any coefficients
  of $\vz^C$ to zero, by the first part of this Lemma, $\vz^0 + \vz^C
  + \vz^D$ cannot be a basic solution unless there is another $r$ and
  $s$ satisfying all qualities of the lemma. A symmetric argument
  holds replacing $r$ with $s$ and vice versa.
\end{proof}

\begin{figure}[ht!]
  \centering
  \includegraphics[scale=0.7]{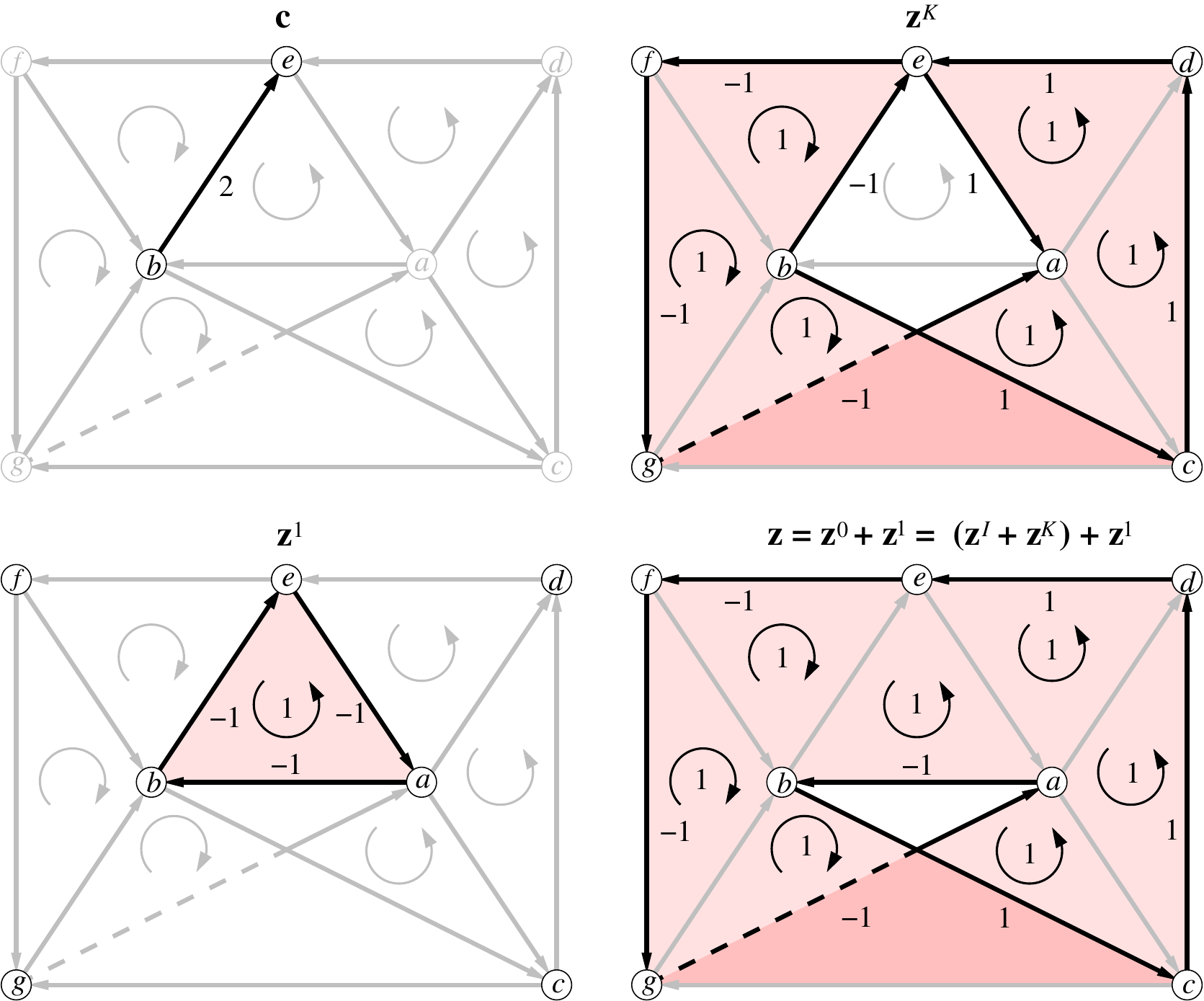}
  \caption{A M\"obius strip illustrating
    Lemma~\ref{lem-must_cancel_two}.}
  \label{fig-nbasic_2_basic}
\end{figure} 

We illustrate the Lemma in Figure \ref{fig-nbasic_2_basic}.  In this
example, $r$ and $s$ correspond to the edges $be$ and $ea$,
respectively, of $\vz^1$. Note that both of these edges have
coefficients of $-1$ in $\vz^1$, but one has a coefficient of $1$ in
$\vz^K$.  Therefore the inequality of the ratios specified in Lemma
\ref{lem-must_cancel_two} holds; one ratio is $1$, and the other $-1$.
Isolating the edges corresponding to $r$ and $s$ as inputs with
coefficients taken from $\vz^0$ will show that the other requirements
of Lemma~\ref{lem-must_cancel_two} are also met.

The remaining results of this section describe the relationship
between the existence of (non)integral basic solutions of an OHCP LP,
and the existence of (non)integral vertices.  
\begin{lemma}
  \label{lem-equiv_coll_conc}
  Let $\vz^0$ be concise.  For any $\vz$, there exists a $\vz'$ such
  that $\{\vz', \vz^0\}$ is linearly concise, and $\vz$ and $\vz'$ are
  equivalent.
\end{lemma}
\begin{proof}
  For each $i$ such that $z_i \neq 0, z_{-i}^0 \neq 0$, subtract $z_i$
  from both $z_i$ and $z_{-i}$.  The result will be both equivalent to
  $\vz'$, and form a linearly concise set with $\vz^0$.
\end{proof}

\begin{lemma}
  \label{lem-equiv_basic}
  $\vz$ is a basic solution if and only if $\vz$ is concise, and each
  $\vz'$ that is concise and equivalent to $\vz$ is a basic solution.
\end{lemma}

\begin{proof}
  Assume $\vz$ is a basic solution.  Then by
  Lemma~\ref{lem-basic_sol_concise}, it is concise.  For any $\vz'$
  that is also concise and equivalent to $\vz$, we transform $\vz$ to
  $\vz'$ by taking each $z_i$ where $z_i \neq z_i'$, and subtracting
  $z_i$ from both $z_i$ and $z_{-i}$.  For each pair of subtractions,
  we make exactly one inequality constraint active, namely $z_i \geq
  0$, and exactly one inactive, $z_{-i} \geq 0$.  For any linearly
  independent set of constraints containing $z_i \geq 0$, if we
  replace this constraint with $z_{-i} \geq 0$, the result must again
  be a linearly independent set.

  If each $\vz'$ that is concise and equivalent to $\vz$ is a basic
  solution, and $\vz$ is concise, then a similar logic holds to show
  $\vz$ is a basic solution.
\end{proof}

\begin{corollary}
  \label{cor-equiv_vert}
  If $\vz$ is a basic solution, then there is a unique vertex $\vz'$
  that is equivalent to $\vz$.
\end{corollary}
\begin{proof}
  For each $i$ with $z_i < 0$, subtract $z_i$ from both $z_i$ and
  $z_{-i}$.  The result will be both concise and equivalent to $\vz$,
  and so by Lemma~\ref{lem-equiv_basic} will be a basic solution.  It
  will also be nonnegative, and therefore a vertex.  Also for each
  $i$, there is one value we may add to both $z_i$ and $z_{-i}$ such
  that the result will be a vertex and equivalent to $\vz$.
\end{proof}

\begin{corollary}
  \label{cor-int_iff_int}
  Let $\vz$ be concise.  Then $\vz$ is integral if and only if each
  $\vz'$ that is concise and equivalent to $\vz$ is integral.
\end{corollary}

\begin{proof}
  We transform $\vz$ to $\vz'$ using the same method as in
  Lemma~\ref{lem-equiv_basic}.  The result holds by the closure of
  integers under addition.
\end{proof}

\section{Fractional Solutions to the OHCP LP, and Elementary Chains}
\label{sec-Elemch}
Consider the special case of OHCP where the input chain $\vc$ is the
{\em elementary} $p$-chain $\ve_i$ for some $i \leq m$, i.e., the
$i\ord$ $p$-simplex has a coefficient of $1$ in the chain while all
other entries are zero.  We refer to this instance of the OHCP LP as
OHCP$_i$, and its feasible region as $P_i$. We analyze the
relationship between the existence of nonintegral values in basic
solutions of the OHCP LP, and the existence of nonintegral values in a
basic solution of OHCP$_i$ for some $i$.
\begin{lemma}
  \label{lem-scalar_mult_basic}
  Let $\alpha \in \left(\R \setminus \{0\}\right)$.  $\vz$ is a basic
  solution of the OHCP LP $O$ with input chain $\vc$ if and only if
  $\alpha\vz$ is a basic solution of the OHCP LP $O^\prime$ with input
  chain $\alpha\vc$.
\end{lemma}

\begin{proof}
  If $\alpha \neq 0$, then multiplying $\vz$ by $\alpha$ does not
  change which entries are nonzero.  Therefore the set of active
  inequality constraints are the same at $\vz$ and $\alpha\vz$.  And
  if we multiply both a solution and the input chain by the same
  scalar, the set of active equality constraints cannot change.
  Therefore the set of active constraints is the same for $\vz$ in $O$
  and for $\alpha\vz$ in $O^\prime$.  Therefore $\vz$ is a basic
  solution to $O$ if and only if $\alpha\vz$ is a basic solution to
  $O^\prime$.
\end{proof}

\begin{theorem}
  \label{thm-if_v_then_V}
  Let $\vz \in P_A$ be a basic solution to the OHCP LP, with $\{\vz,
  \vz^I\}$ linearly concise. There exists some matrix $Z$ such that
  the columns of $Z$ form a linearly concise set, each column $Z_i$ of
  $Z$ is a basic solution in $(P_i)_A$, and $Z\vc = \vz$.
\end{theorem}

\begin{proof}
  Assume $\vz$ is a basic solution in $P_A$.  Begin constructing $Z$
  by first setting each column $Z_i$ to the vector that is linearly
  concise with $\vz^I$, and equivalent to the identity solution of
  OHCP$_i$.  Then $Z\vc = \vz^I$.  By Theorem~\ref{thm-vk_active} and
  Lemma~\ref{lem-vk_equiv_bi}, the identity solution to any OHCP LP is
  a basic solution.  So by Lemma~\ref{lem-equiv_basic}, each $Z_i$ is
  a basic solution to OHCP$_i$ in $(P_i)_A$.  Also, by Lemma
  \ref{lem-scalar_mult_basic}, $c_iZ_i$ is a basic solution to the
  OHCP LP with input chain $c_i\ve_i$.

  Let $\vz^K = \vz - \vz^I$. If $\vz^K \neq \vzero$, distribute it
  among the columns of $Z$ according to the following algorithm.
  \begin{enumerate}
  \item Let $\vz^0$ = $\vz^I$.
  \item \label{stp-choose_i} Since $\{\vz, \vz^0\}$ is linearly
    concise, $\vz^K$ is concise. By
    Lemma~\ref{lem-last_OHCP_term_cancels}, there is some
    $x$-coordinate $i$ such that $z_i^0 \neq 0, z_i= 0$.  It must also
    be true that $c_{i'} \neq 0$ where $i'$ is either $i$ or $-i$.
    Note that $Z_{i'}$ is equal to $\begin{bmatrix} \ve_{i'}^T &
      \vzero^T & \vzero^T & \vzero^T
    \end{bmatrix}^T$.
  \item IF $c_{i'}Z_{i'} + \vz^K$ is not a solution to the OHCP LP
    with input chain $c_{i'}\ve_{i'}$, THEN
    \begin{enumerate}
    \item Construct $\vz^C$ and $\vz^D$ according to the algorithm in
      Theorem~\ref{thm-v_plus_vk_basic}.  Note that $\vz^0$ and
      $\vz^K$ represent the same vectors in
      Theorem~\ref{thm-v_plus_vk_basic} as they do here.  However, the
      difference between $\vz$ of this Lemma and $\vz$ of
      Theorem~\ \ref{thm-v_plus_vk_basic} is $\vz^I - c_{i'}Z_{i'}$.
    \item $\vz^C$ cannot be $\vzero$.  Otherwise, $\vz^D$ would be
      $\vz^K$, and so by our choice of $i$ in step~\ref{stp-choose_i},
      would not satisfy Property~\ref{cnd-v_contains_vr} of
      Theorem~\ref{thm-v_plus_vk_basic}.  Also recall from
      Theorem~\ref{thm-v_plus_vk_basic} that all vectors referred to
      in this algorithm form a linearly concise set.  Add
      $\frac{1}{c_{i'}}\vz^C$ to $Z_{i'}$.
    \item Set $\vz^K := \vz^D$.  By Property~\ref{cnd-vr_nz} of
      Theorem~\ref{thm-v_plus_vk_basic}, we still have $\vz^K \neq 0$.
    \item Set $\vz^0 := \vz^0 + \vz^C$.  $\{\vz, \vz^0\}$ is
      still linearly concise.  And because $\vz^C$ satisfies
      Property~\ref{cnd-no_y_elim} of
      Theorem~\ref{thm-v_plus_vk_basic}, we still have that for each
      $y$-coordinate $j, z_j^0 \neq 0 \implies z_j \neq 0$.  Also, by
      Property~\ref{cnd-v_contains_vr}, for any $x$-coordinate $i$
      where $z_i^0 \neq 0$ and $z_i = 0, c_{i'} \neq 0$ where $i'$ is
      either $i$ or $-i$.  This is because the only $x$-coefficients
      that are not zero in the $\vz$ of
      Theorem~\ref{thm-v_plus_vk_basic}, but are zero in the $\vz$ of
      this theorem must be nonzero in $\vz^I$.
    \item LOOP to Step~\ref{stp-choose_i}. Note that the next $i$
      chosen in Step~\ref{stp-choose_i} cannot be the same as any
      previous $i$s chosen, since by
      Lemma~\ref{lem-last_OHCP_term_cancels}, and our previous choices
      of $i$, $z_i^0$ must be 0 for each previous $i$.
    \end{enumerate}
  \item Add $\frac{1}{c_{i'}}\vz^K$ to $Z_{i'}$.
  \item STOP.
  \end{enumerate}
  Because each $\vz^C$ is nonzero, and $K$ is finite, this algorithm
  must terminate, giving us the desired $Z$.
\end{proof}

Consider the case of $p = 1$. Because a basic solution $\vz$ cannot be
decomposed as in Theorem~\ref{thm-v_plus_vk_basic} the set of
triangles of $\vz$ form a union of 2D spaces each of which must be
connected to at least one edge $\tau$ of the input chain $\vc$. Then
each of these 2D spaces must be a basic solution to the OHCP LP where
$\tau$ has the only nonzero coefficient in the input chain $\vc_\tau$,
and the coefficient of $\tau$ is the same in $\vc$ and $\vc_\tau$.  By
Lemma~\ref{lem-scalar_mult_basic}, we may scale each of these basic
solutions as necessary to be basic solutions of elementary
chains. These basic solutions of elementary chains, adjusted as
necessary to form a linearly concise set, are columns of $Z$.  Note
that the construction of $Z$ may not be unique, and many of the
columns of $Z$ may be equivalent to identity solutions.

We may refer to Figure~\ref{fig-Basic_Sol_Ex_1} and think of the edges
of the last input chain as three different input chains consisting
individually of $ab, bc$, and $ca$. We may then decompose the basic
solution shown into solutions equivalent to identity solutions for
$bc$ and $ca$, and a basic solution to $ab$.
\begin{lemma}
  \label{lem-frac_basic}
  For a given complex $K$, there is an OHCP LP with integral input
  chain $\vc$ that has a nonintegral basic solution if and only if
  there is some $i$ such that the $i\ord$ coefficient is nonzero in
  $\vc$, and OHCP$_i$ has a nonintegral basic solution.
\end{lemma}
\begin{proof}
  If there is some $i$ such that OHCP$_i$ has a nonintegral basic
  solution, then $\ve_i$ is an integral input chain that has a
  nonintegral basic solution.  If $\vz$ is a nonintegral basic
  solution to the OHCP LP with integral input chain $\vc$, then by
  Theorem~\ref{thm-if_v_then_V}, $\vz$ is the sum of $n$ vectors each
  of the form $c_i\left(\vz^I\right)_{\pm i} + \vz^C$, where $\vz^C
  \neq \vzero \implies c_i \neq 0$.  Since integers are closed under
  addition, one of these terms must be nonintegral.  Since $c_i$ and
  $\left(\vz^I\right)_{\pm i}$ both must be integral, there must be
  some $\vz^C$ that is nonintegral.  Therefore some
  $\frac{1}{c_i}\vz^C$ is also nonintegral, and so
  $\left(\vz^I\right)_{\pm i} + \frac{1}{c_i}\vz^C$ is nonintegral.
  By Theorem~\ref{thm-if_v_then_V}, each $\left(\vz^I\right)_{\pm i} +
  \frac{1}{c_i}\vz^C$ is a basic solution to OHCP$_i$.  Therefore some
  $\left(\vz^I\right)_{\pm i} + \frac{1}{c_i}\vz^C$ is a nonintegral
  basic solution to OHCP$_i$.  Furthermore, the $i\ord$ coefficient of
  $\vc$ must be nonzero, otherwise this sum is undefined.
\end{proof}
\begin{lemma}
  \label{lem-comm_sol}
  Let $\vz^0$ be concise and be in the hyperplane $P_A$ of the OHCP
  $O$ with input chain $\vc$, with $\vz$ a basic solution to the same
  OHCP in $P_A$ where for each $y$-coordinate $j, z_j = 0 \implies
  z_j^0 = 0$.  Let $\vy^0$ be the $2(m+n)$-vector with all
  $y$-coefficients equal to those of $\mathbf{z}^0$, and all
  $x$-coefficients 0.  Then $\vz - \vy^0$ is a basic solution of the
  OHCP $O^0$ with input chain $\vc^0$ in $(P_A)^0$ where
  $[\left(\vc^0\right)^T \vzero^T]^T$ is equivalent in all
  $x$-coordinates to $\vz^0$.
\end{lemma}
\begin{proof}
  Since $\vz$ and $\vz^0$ are both in $P_A$, the difference between
  them is in $\Ker(A)$.  Since $[\left(\mathbf{c}^0\right)^T
    \vzero^T]^T$ and $\vz^0$ are equivalent in all $x$-coordinates,
  $\vz^0 - \vy^0$ is equivalent to the identity solution for $O^0$.
  Therefore $\vz^0 - \vy^0$ and $\vz - \vy^0$ are both in $(P_A)^0$.

  We show that $\vz - \vy^0$ is a basic solution of $O^0$ by
  contrapositive.  Let $\vz^*$ represent $\vz - \vy^0$.  Suppose
  $\vz^*$ is not a basic solution of $O^0$.  Then by
  Theorem~\ref{thm-vk_active}, $\exists \vz^K \in \left(\Ker(A)
  \setminus \{\vzero\}\right)$ where $z_i^* = 0 \implies z_i^K =
  0.\ \vz$ and $\vz^*$ agree in all $x$-coordinates.  We also have for
  each $y$-coordinate $j, z_j = 0 \implies z_j^0 = 0 \implies
  y_j^0 = 0 \implies z_j^* = 0 \implies z_j^K = 0$.  Therefore we
  have $z_i = 0 \implies z_i^K = 0$ for all $i$.  Therefore by
  Theorem~\ref{thm-vk_active}, $\vz$ is not a basic solution of $O$.
\end{proof}

For $p = 1$, due to Corollary~\ref{cor-equiv_vert},
Lemma~\ref{lem-comm_sol} is saying that if we have a set of edges
$\vx$ that is a vertex for some OHCP LP with input chain $\vc$, and we
transform $\vc$ to $\vx$ by adding triangles one at a time, or at
least not eliminating triangles previously added, then $\vx$ will be a
vertex to any OHCP LP that has as input any of these intermediate edge
sets we get at each step of this transformation.
\begin{lemma}
  \label{lem-all_y_fract}
  For a given complex $K$, there is an OHCP LP with integral input
  chain $\vc$ that has a nonintegral basic solution if and only if
  there is some $i$ such that OHCP$_i$ has a basic solution where all
  $y$-coordinates that are nonzero are nonintegral.
\end{lemma}
\begin{proof}
  Let $\vz$ be a basic solution with nonintegral coefficients for an
  OHCP LP with integral input chain $\vc$.  By
  Theorem~\ref{thm-if_v_then_V}, there is some matrix $Z$ such that
  the columns of $Z$ form a linearly concise set, each column $Z_i$ of
  $Z$ is a basic solution in $(P_i)_A$, and $Z\vc = \vz$.  By a
  similar logic as Lemma~\ref{lem-frac_basic}, one of these columns is
  nonintegral, and this column is of the form $\left(\vz^I\right)_{\pm
    i} + \frac{1}{c_i}\vz^C$.  For this $\vz^C$, let $\mathscr{J}$ be
  the set of $y$-coordinates with nonzero integral coefficients in
  $\frac{1}{c_i}\vz^C$.

  If $\mathscr{J}$ is empty, then we have our desired result.  If
  $\mathscr{J}$ is nonempty, decompose $\vz^C$ into its linear
  combination of basis vectors of $\Ker(A)$ that are equivalent to the
  last $n$ columns of the matrix $N$ of Lemma~\ref{lem-vk_equiv_bi}.
  Let $\vz^0$ be the sum of $\ve_{\pm i}$ and the components of this
  linear combination with nonzero element in $\mathscr{J}$.  Note that
  $\vz^0$ is integral.  Let $\vc^0$ be the input chain where
  $[\left(\mathbf{c}^0\right)^T \vzero^T]^T$ is equivalent in all
  $x$-coordinates to $\vz^0$. Let $\vy^0$ be the $2(m+n)$-vector such
  that $j \in \mathscr{J} \implies y_j^0 = (1/c_{i})z_j^C, j \notin
  \mathscr{J} \implies y_j^0 = 0$.  Then by Lemma~\ref{lem-comm_sol},
  $\left(\vz^I\right)_{\pm i} + (1/c_i)\vz^C - \vy^0$ is a basic
  solution to the OHCP LP $O^0$ with input chain $\vc^0$, and all of
  its nonzero $y$-coefficients are nonintegral.

  Now let $\vz = \left(\vz^I\right)_{\pm i} + (1/c_i)\vz^C - \vy^0$,
  let $\vc = \vc^0$, and apply the same logic as above.  Note that
  with each iteration of this process, we lessen the number of nonzero
  $y$-coefficients; the set of nonzero $y$-coefficients of $\vz$
  contains the set of nonzero $y$-coefficients in $(1/c_i)\vz^C$, and
  this set decreases in size each round.  Because $K$ is finite, this
  process must eventually terminate.
\end{proof}

\begin{remark} \label{rem-elemchn}
  Lemmas \ref{lem-frac_basic} and \ref{lem-all_y_fract} allow us to
  concentrate on the easier to analyze case of elementary input chains
  for the OHCP LP in order to arrive at our main results.
\end{remark}


\section{Projections Onto the Space of $p$-Simplex Coefficients}
\label{sec-Projection_to_X}
Call the space of $x$-variables $\mathscr{X} \left(= \R^{2m}\right)$.
We study the projections of the OHCP LP, $P, P_A, \Ker(A),$ basic
solutions, and vertices onto $\mathscr{X}$. For any of these objects
$\Omega$, let $\Omega\toX$ represent the projection of $\Omega$ onto
$\mathscr{X}$. Since the $y$-variables do not appear in the OHCP LP
objective function, a vertex of $P\toX$ must be optimal. We establish
correspondences between the basic solutions in the projection space
and those in the original space. We then prove equivalent
relationships between OHCP and OHCP$_i$ in the projection space. 

First, we extend the definition of concise to $(2m)$-vectors in the
natural way.  We also rephrase the definition of a basic solution
below.
\begin{definition}
  \label{def-basic_sol_X}
  Let $\mathscr{C}$ be the set of constraints of an OHCP LP $O$ that
  are not orthogonal to $\mathscr{X}$, i.e., the set of constraints
  with a nonzero $x$-coefficient.  For any $\vx \in \mathscr{X}$, let
  $\mathscr{C}_{\vx}$ be the set of elements of $\mathscr{C}$ active
  at $\vx$. Then $\vx$ is a \emph{basic solution of $O\toX$} if and
  only if there is no other point in $\mathscr{X}$ where all elements
  of $\mathscr{C}_{\vx}$ are active.
\end{definition}

Note that if $\vx$ fails Definition~\ref{def-basic_sol_X}, then
$\mathscr{C}_{\vx}$ must be active for an entire affine space of some
dimension $d \geq 1$.  To say $\vx \in X$ is feasible in $O\toX$ is
equivalent to saying $\vx \in P\toX$. We justify
Definition~\ref{def-basic_sol_X} and how it relates to vertices of
$P\toX$ with the following lemma.
\begin{lemma}
  \label{lem-bas_feas_sol_vert}
  $\vx \in \mathscr{X}$ is a basic feasible solution of $O\toX$ if and
  only if $\vx$ is feasible, and not a convex combination of any two
  distinct elements of $P\toX$.
\end{lemma}
\begin{proof}
  We prove both directions by contrapositive. If $\vx$ is infeasible,
  then clearly it is not a basic feasible solution.  Suppose $\vx$ is
  a convex combination of two distinct elements $\vx^1$ and $\vx^2$ of
  $P\toX$. Then $\vx = \lambda \vx^1 + (1 - \lambda) \vx^2$ for some
  $\lambda \in (0, 1)$.  Since $\vx^1, \vx^2 \in P\toX$, they are
  nonnegative, and so the set of nonzero coefficients of $\vx$ is the
  union of the sets nonzero coefficients of $\vx^1$ and $\vx^2$.
  Therefore any inequality constraints in $\mathscr{C}$ active at
  $\vx$ must also be active at $\vx^1$ and $\vx^2$. Since $P$ is
  convex, $P\toX$ is convex.  Since $\vx^1, \vx^2 \in P\toX, \vx \in
  P\toX$.  Therefore all equality constraints are active at all three
  points.  Therefore $\vx$ does not satisfy Definition
  \ref{def-basic_sol_X} as a basic solution of $O\toX$.

  Now suppose $\vx$ is feasible, but does not satisfy the conditions
  specified in Definition \ref{def-basic_sol_X}.  Let $\vx^\prime$ be
  another point in $\mathscr{X}$ where all elements of
  $\mathscr{C}_{\vx}$ are active.  All equality constraints are active
  at $\vx$, and so are also active at $\vx^\prime$.  Therefore
  $\vx^\prime \in \left(P\toX\right)_A$.  Let $L$ be the line in
  $\mathscr{X}$ containing $\vx$ and $\vx^\prime$.  Then $L \subset
  \left(P\toX\right)_A$.  Any point in $L$ may be expressed as $\alpha
  \vx^\prime + (1 - \alpha) \vx$ for some $\alpha \in \R$.  Choosing a
  value for $\alpha$ defines a point.  Because $P\toX$ is convex,
  there is at most two values for $\alpha$ such that $\alpha
  \vx^\prime + (1 - \alpha) \vx$ is on the boundary of $P\toX$. Since
  $\vx$ is defined by $\alpha = 0$, and $\vx$ is feasible, at most one
  such value is positive, and at most one such value is negative.

  If there is no such value for $\alpha$, then $L \subset P\toX$.
  Then choose values 1 and $-1$ for $\alpha$, to define $\vx^1$ and
  $\vx^2$.  If there is only one such value $\alpha$, then choose
  $\alpha$ and $-\alpha$ to define $\vx^1$ and $\vx^2$.  If there are
  two such values $\alpha_1, \alpha_2$, then choose $\min
  \{\abs{\alpha_1}, \abs{\alpha_2}\}$ and $-\left(\min
  \{\abs{\alpha_1}, \abs{\alpha_2}\}\right)$ to define $\vx^1$ and
  $\vx^2$.  In any case, both $\vx^1$ and $\vx^2$ are in $P\toX$, and
  $\vx = (1/2)\vx^1 + (1/2)\vx^2$.
\end{proof}

Lemma~\ref{lem-bas_feas_sol_vert} shows we may define vertices of
$P\toX$ the same way as in $P$. We now show results for basic
solutions of $O\toX$ and vertices of $P\toX$ that parallel many of our
previous results. Proofs are omitted where the logic is a natural
parallel of these previous results.

\begin{corollary}
  \label{cor-vk_active_X}
  Let $\vx \in \left(P_A\right)\toX$.  $\vx$ is a basic solution of
  $O\toX$ if and only if $\,\forall \vx^K \in \left(\Ker(A)\toX
  \setminus \{\vzero\}\right), \ \exists i: x_i = 0, x_i^K \neq 0.$
\end{corollary}

\begin{lemma}
  \label{lem-basic_sol_concise_X}
  Any basic solution $\vx$ of $O\toX$ is concise.
\end{lemma}

A significant difference between the basic solutions in the projection
space and those in the original space is that in the projection,
because they are determined only by $p$-coefficients, the choice of
input chain on a complex has no impact on whether or not a solution in
$\left(P_A\right)\toX$ is basic in the projection.  Referring back to
Figure~\ref{fig-Basic_Sol_Ex_1}, the last solution is not basic in the
projection. Lemma~\ref{lem-comm_sol_X} formalizes this idea of the
input chain not mattering, with some added context useful for our main
result.  It is a parallel to Lemma~\ref{lem-comm_sol}, but becomes
simpler in the projection space.
\begin{lemma}
  \label{lem-comm_sol_X}
  Let $\vz^0$ be concise and be in the hyperplane $P_A$ of the OHCP
  $O$ with input chain $\vc$, with $\vz$ a basic solution to the same
  OHCP in $P_A$ and $\vz\toX$ a basic solution to $O\toX$. Then
  $\vz\toX$ is a basic solution of the OHCP $O^0\toX$ with input chain
  $\vc^0$ in $(P_A)^0\toX$ where $[\left(\vc^0\right)^T \vzero^T]^T$
  is equivalent to $\vz^0\toX$.
\end{lemma}
\begin{proof}
  Since $\vz$ an $\vz^0$ are both in $P_A$, the difference between
  them is in $\Ker(A)$.  Since $[\left(\mathbf{c}^0\right)^T
    \vzero^T]^T$ and $\vz^0$ are equivalent in all $x$-coordinates,
  $\vz^0\toX$ is equivalent to the identity solution for $O^0$.
  Therefore $\vz^0\toX$ and $\vz\toX$ are both in $(P_A)^0\toX$.

  We show that $\vz\toX$ is a basic solution of $O^0\toX$ by
  contrapositive.  Let $\vx = \vz\toX$. Suppose $\vx$ is not a basic
  solution of $O^0\toX$.  Then by Corollary~\ref{cor-vk_active_X},
  $\exists \vx^K \in \left(\Ker(A)\toX \setminus \{\vzero\}\right)$
  where $x_i = 0 \implies x_i^K = 0$.  But then by
  Corollary~\ref{cor-vk_active_X}, $\vz\toX$ is not a basic solution
  of $O\toX$.
\end{proof}
\begin{lemma}
  \label{lem-pre-im_vert_has_vert}
  For any basic solution $\vx \in \left(P_A\right)\toX$ of $O\toX$,
  there is a basic solution $\vz \in P_A$ of $O$ where $\vz\toX =
  \vx$.
\end{lemma}
\begin{proof}
  Since $\vx \in \left(P_A\right)\toX$, there is some $\vz'$ where
  $\vz'\toX = \vx$, and $\vz' \in P_A$.  Since $\vz'$ and $\vx$ agree
  in all $x$-variables and $\vx$ is a basic solution, for any $\vz^K$
  where $z_i = 0 \implies z_i^K = 0, \vz^K\toX$ must be $\vzero$.  If
  such a $\vz^K$ exists, we may use the algorithm of
  Theorem~\ref{thm-v_plus_vk_basic} to arrive at a basic solution
  $\vz$ of $O$ in $P_A$ where $\vz\toX = \vx$.
\end{proof}

\begin{corollary}
  \label{cor-noint-impl-noint}
  If for a given complex $K$, all vertices of any OHCP LP are
  integral, then all vertices of any projection of an OHCP LP onto
  $\mathscr{X}$ must be integral.
\end{corollary}
\begin{proof}
  We prove by contrapositive.  If there exists some $O\toX$ with a
  nonintegral vertex $\vx$, then by
  Lemma~\ref{lem-pre-im_vert_has_vert} and
  Corollary~\ref{cor-equiv_vert}, there is a vertex $\vz$ of $O$ where
  all $x$-coordinates of $\vz$ and $\vx$ agree.  Therefore $\vz$ is
  nonintegral.
\end{proof}

\begin{corollary}
  \label{cor-last_OHCP_term_cancels_X}
  Let $\vx \in \left(P_A\right)\toX$ be a basic solution of
  $O\toX$. Let $\vx^0 \in \left(P_A\right)\toX$ with $\vx^0$ concise.
  Let $\vx = \vx^0 + \vx^K$, with $\vx^K$ being concise. Then $\vx^K
  \neq \vzero$ if and only if there exists a $p$-coefficient that is
  zero in $\vx$, but nonzero in $\vx^0$.
\end{corollary}
\begin{corollary}
  \label{cor-v_plus_vk_basic_X}
  Let $\vx^0 \in \left(P_A\right)\toX$ be a basic solution of $O\toX$.
  Let $\vx^K \in \Ker(A)\toX$ with $\{\vx^0, \vx^K\}$ linearly
  concise.  Let $\vx = \vx^0 + \vx^K$.  Then $\vx$ is a basic solution
  if and only if there do not exist $\vx^C, \vx^D$ satisfying the
  following properties:
  \begin{enumerate}
  \item \label{cnd-sum_xk} $\vx^C + \vx^D = \vx^K$.
  \item \label{cnd-in_KerAX} $\vx^C, \vx^D \in \Ker(A)\toX$.
  \item \label{cnd-xd_nz} $\vx^D \neq \vzero$.
  \item \label{cnd-lin_conc_X} $\{\vx^0, \vx^K, \vx^D\}$ is linearly concise.
  \item \label{cnd-x'basic} $\vx^0 + \vx^C = \vx^1$ is a basic
    solution to $O\toX$.
  \item \label{cnd-x_doms_xr}$x_i^D \neq 0 \implies x_i \neq 0\ \forall i$.
  \end{enumerate}
\end{corollary}

\begin{proof}
  We prove both directions again by contrapositive.  The first
  direction follows in the same way as
  Theorem~\ref{thm-v_plus_vk_basic}, replacing $\vz$ with $\vx$ and
  Lemma~\ref{lem-last_OHCP_term_cancels} with Corollary
  \ref{cor-last_OHCP_term_cancels_X}.  Then suppose $\vx$ is not a
  basic solution. $\{\vx^0, \vx^K, \vx\}$ is still linearly concise.
  Construct $\vx^D$ and find $\vx^1$ using the following algorithm.
  \begin{enumerate}
  \item Let $\vx^D = \vzero, \vx^1 = \vx$.  Then $\{\vx^0, \vx^K,
    \vx^D, \vx^1\}$ is linearly concise.
  \item 
    \label{stp-vprime_not_basic_X}
    $\vx^1$ must be in $\left(P_A\right)\toX$, and is not a basic
    solution.  By Corollary~\ref{cor-vk_active_X}, $\exists \, \vz^N
    \in (\Ker(A)\toX \setminus \{\vzero\})$ where $x_i^N \neq 0
    \implies x_i^1 \neq 0$.  Because $\{\vx^0, \vx^K, \vx^D, \vx^1\}$
    is linearly concise, $\{\vx^0, \vx^K, \vx^D, \vx^1,\vx^N\}$ is
    linearly concise.
  \item Find $i$ such that $x_j^N \neq 0 \implies \abs{x_i^1 / x_i^N}
    \leq \abs{x_j^1 / x_j^N}$.
  \item Let $\alpha = x_i^1 / x_i^N$.
  \item Let $\vx^D = \vx^D + \alpha\vx^N, \vx^1 = \vx^1 -
    \alpha\vx^N$.  Because we may add any linear combination of a set
    of linearly concise vectors to the linearly concise set, $\{\vx^0,
    \vx^K, \vx^D, \vx^1,\vx^N\}$ is still linearly concise.
  \item IF $\vx^1$ is not a basic solution of $O\toX$ THEN LOOP to
    Step~\ref{stp-vprime_not_basic_X}.
  \item STOP.
  \end{enumerate}
  Because $K$ is finite, and we make at least one entry zero that was
  nonzero in $\vx^1$ in each loop, and do not make any zero entries in
  $\vx^1$ nonzero, then by Corollary~\ref{cor-vk_active_X}, this
  algorithm must eventually terminate after at most $m$ iterations.
  And as in Theorem \ref{thm-v_plus_vk_basic}, $\vx^D, \vx^1$, and
  $\vx^C = \vx^K - \vx^D$ satisfy all criteria of the Corollary.
\end{proof}

\begin{corollary}
  \label{cor-must_cancel_two_X}
  Let $\vx^0 \in \left(P_A\right)\toX$ be concise.  Let $\vx = \vx^0 +
  \vx^1$.  If $\vx$ is a basic solution in $\left(P_A\right)\toX$,
  then for each $\vx^K \in \left(\Ker(A)\toX \setminus
  \{\alpha\vx^1\}\right)$ for $\alpha \in \R$, where $x_i^K \neq 0
  \implies x_i^0 \neq 0$, there must be two coordinates $r$ and $s$
  where $x_r = x_s = 0$, $x_r^K, x_s^K \neq 0$, and $\displaystyle
  (x_r^K/x_r^1) \neq (x_s^K/x_s^1)$.  Furthermore, if
  $\left(O_r\right)\toX$ and $\left(O_s\right)\toX$ are the
  projections of OHCP LPs with input chains where the only nonzero
  coefficients are $r$ and $s$ respectively, with these coefficients
  equaling those in $\vx^0$, then $\vx^1 + \vx^{I\prime}$ is a basic
  solution to $\left(O_r\right)\toX$ and $\left(O_s\right)\toX$ where
  $\vx^{I\prime}$ is the solution with $\{\vx^1, \vx^{I\prime}\}$
  linearly concise and equivalent to the identity solution for
  $\left(O_r\right)\toX$ and $\left(O_s\right)\toX$, respectively.
\end{corollary}
\begin{corollary}
  \label{cor-equiv_basic_X}
  $\vx$ is a basic solution of $O\toX$ if and only if $\vx$ is
  concise, and each $\vx'$ that is concise and equivalent to $\vx$ is
  a basic solution.
\end{corollary}
\begin{corollary}
  \label{cor-equiv_vert_X}
  For each basic solution $\vx$ of $O\toX$, there is a unique vertex
  $\vx'$ of $O\toX$ that is equivalent to $\vx$.
\end{corollary}
\begin{corollary}
  \label{cor-int_iff_int_X}
  Let $\vx$ be a concise solution of $O\toX$.  Then $\vx$ is integral
  if and only if each solution $\vx'$ of $O\toX$ that is concise and
  equivalent to $\vx$ is integral.
\end{corollary}

\begin{theorem}
  \label{thm_frac_vert_X}
  For a given complex $K$, there is a nonintegral vertex $\vz$ of some
  OHCP LP $O$ with integral input chain $\vc$ 
  where $\vz\toX$ is a vertex of $P\toX$ if and only if there is some
  $i$ such that OHCP$_i$ has a nonintegral vertex $\vz'$ such that
  $\vz'\toX$ is a vertex of $P_i\toX$ where the $i\ord$ coefficient of
  $\vc$ is nonzero.
\end{theorem}
\begin{proof}
  One direction of the if and only if is trivially true: if there is
  some $i$ and OHCP$_i$, then the more general case of OHCP LP follows
  immediately.  We prove the other direction by contrapositive.  If
  there is no such $i$ where OHCP$_i$ has a nonintegral vertex $\vz'$,
  then by Lemma~\ref{lem-frac_basic} and
  Corollary~\ref{cor-equiv_vert}, there can be no nonintegral vertex
  $\vz$ of any OHCP LP.

  Now suppose there is an $i$ where OHCP$_i$ has a nonintegral vertex
  $\vz'$, but no such $i$ where $\vz'\toX$ is a vertex of $P_i\toX$.
  Then for any nonintegral vertex $\vz$ of an OHCP LP, by
  Theorem~\ref{thm-if_v_then_V} and Lemma~\ref{lem-frac_basic}, some
  column $\vz'$ of $Z$ with $Z\vc = \vz$ is a nonintegral basic
  solution of OHCP$_i$ for some $i$. If we construct $Z$ using the
  algorithms of Theorem~\ref{thm-if_v_then_V} and
  Theorem~\ref{thm-v_plus_vk_basic}, then by
  Condition~\ref{cnd-v_contains_vr} of
  Theorem~\ref{thm-v_plus_vk_basic}, any nonzero $p$-coefficient of
  $\vz'$ is nonzero in $\vz$.

  Since $\vz'\toX$ cannot be a vertex, it cannot be a basic solution
  of OHCP$_i\toX$.  So by Corollary~\ref{cor-v_plus_vk_basic_X}, there
  is some $\vx^D \in \Ker(A)\toX$ where every nonzero coefficient of
  $\vx^D$ is nonzero in $\vz'\toX$.  Since all these coefficients are
  $p$-coefficients, all the nonzero coefficients of $\vx^D$ are
  nonzero in $\vz\toX$.  So by Corollary~\ref{cor-v_plus_vk_basic_X},
  $\vz\toX$ cannot be a basic solution to $O\toX$, and so is not a
  vertex of $P\toX$.
\end{proof}

\begin{corollary}
  \label{cor-frac_vert_X}
  For a given complex $K$, there is an OHCP LP with integral input
  chain $\vc$ 
  that has a nonintegral vertex $\vz$ where $\vz\toX$ is a vertex of
  $P\toX$ if and only if there is some $i$ such that OHCP$_i$ has a
  nonintegral vertex $\vz'$ with all nonzero $y$-coefficients
  nonintegral where $\vz'\toX$ is a vertex of $P_i\toX$.
\end{corollary}
\begin{proof}
  The result follows from Lemmas~\ref{lem-all_y_fract}
  and~\ref{lem-comm_sol_X}, and
  Corollaries~\ref{cor-equiv_vert},~\ref{cor-int_iff_int},
  \ref{cor-equiv_vert_X}, and~\ref{cor-int_iff_int_X}.
\end{proof}

\vspace*{-0.15in}
\section{Minimally Non Totally-Unimodular Submatrices of $\Bdmat$}
\label{sec-BGs_of_MNTUs}

We study minimal violations of total unimodularity, and describe a
stricter version of a minimal violation submatrix of the boundary
matrix. Intuitively, we study M\"obius strips that do not contain
smaller M\"obius strips within their triangles. A \emph{minimally non
  totally-unimodular} (MNTU) matrix is a matrix $M$ that is not
totally unimodular, but every proper submatrix of $M$ is totally
unimodular (also referred to as {\em almost totally unimodular}
matrices \cite{Ca1965}). Several properties of an MNTU matrix $M$ are
known previously \cite{Ca1965,Truemper1992VII}.
\begin{enumerate}
\setlength{\itemsep}{-0.017in}
\item \label{prp-NTU_has_MNTU} A matrix is not totally unimodular if
  and only if has an MNTU submatrix $M$.
\item \label{prp-MNTU_det} $\det(M) = \pm 2$.
\item \label{prp-MNTU_Eulr}Every column and every row of $M$ has an
  even number of nonzero entries, i.e., $M$ is \emph{Eulerian}.
\item \label{prp-MNTU_sum_2} The sum of the entries of $M$ is $2 \bmod
  4$.
\item \label{prp-MNTU_bodd_circuit} The bipartite graph representation
  of $M$ is a chordless (i.e., induced) circuit \cite{CoRa1987}.
\end{enumerate}
The bipartite graph representation \cite{Ca1965,CoRa1987} a submatrix
$M$ of $\Bdmat$ has a vertex for each row and for each column of $M$,
and an undirected edge for each nonzero entry $M_{ij}$ connecting the
vertices for row $i$ and column $j$. Notice that each edge connects a
row vertex, or $p$-vertex, with a column vertex, or $q$-vertex. A
circuit $C$ in a weighted graph is {\em b-odd} ({\em b-even}) if the
sum of the weights of the edges in $C$ is $2 \bmod 4 $ ($0 \bmod
4$). The quality of $C$ being b-even, b-odd, or neither is called the
b-parity of $C$. The following theorem characterizes this bipartite
graph as a circuit.
\begin{theorem}
  \label{thm-induced_flag_trav}
  Given a circuit $C$ that is the bipartite graph representation of an
  MNTU submatrix $M$ of $\Bdmat$, and a set of flags placed on an
  arbitrary subset of the $q$-vertices of $C$, there exists a
  traversal of $C$ such that each portion of the traversal of $C$
  between two consecutive flags is induced.
\end{theorem}

\begin{proof}
  Recall that for any circuit $C$, if an edge $h \in C$ is a potential
  chord of a subgraph of $C$, then both of its end points are of
  degree $4$ or more in $C$. If there is no path in $C$ between two
  flags, then every path between them must contain both end points of
  a chord, and we have a set of paths like one of the graphs in
  Figure~\ref{fig-flag_graphs}.  The graphs shown are abstractions of
  $C$ in the case where the end points of $h_1$ and $h_2$ all have
  degree 4, which is the minimum possible degree for these vertices.
  Graph $C_1$ is the case showing the four half-paths, and the other
  graphs show the possibilities of how these half-paths can connect.
  \begin{figure}[ht!]
    \centering
    \includegraphics[scale=1]{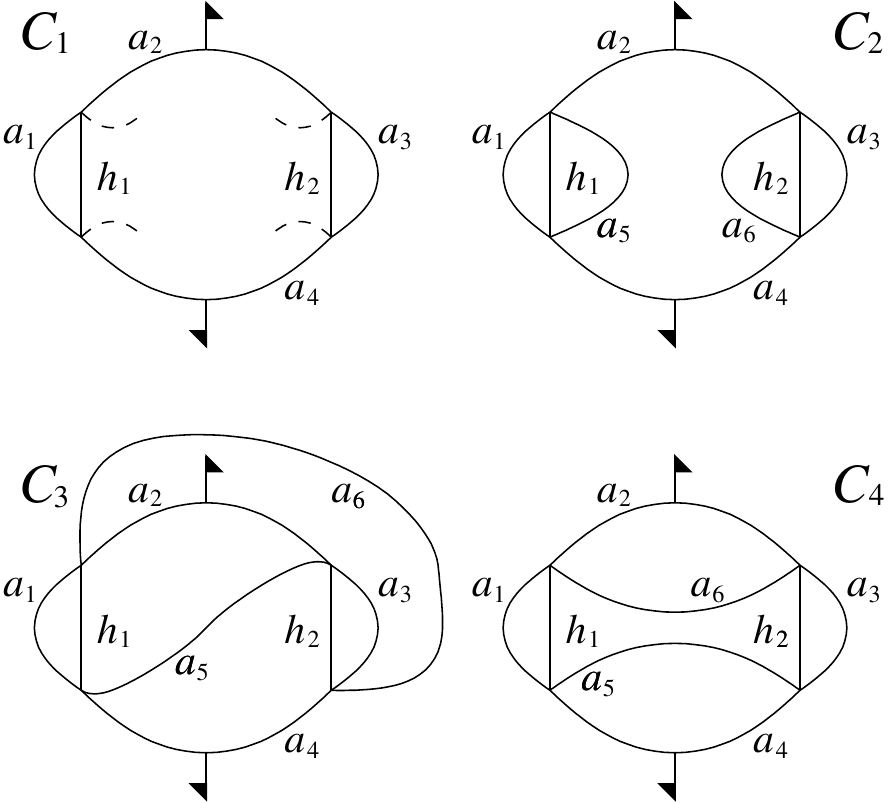}
    \caption{Abstract representations of $C$.}
    \label{fig-flag_graphs}
  \end{figure} 

  In graph $C_2$, though any path between the two flags contains both
  ends of a chord, we may still traverse the entire graph in such a
  way that each portion of the traversal between flags is induced.
  This can be done by traversing $a_1$ and $a_5$ immediately before or
  after $h_1$, and traversing $a_3$ and $a_6$ immediately before or
  after $h_2$.  In this way, all paths between end points of a
  specific chord that do not contain a flag are traversed
  consecutively.

  In graph $C_3$, we may also traverse the graph in such a way that
  each portion of the graph is induced.  If we traverse either $a_5$
  or $a_6$ as soon as possible, then we have a path between the flags
  that does not contain both end points of any chord.  We may then
  traverse that remainder of the graph, which is also induced.

  In graph $C_4$, we cannot traverse the graph without one of the
  portions between the flags not being induced.  This graph may be
  decomposed into four cycles: $a_1$ and $h_1$, $a_2$ and $a_6$, $a_3$
  and $h_2$, as well as $a_4$ and $a_5$.  Note that of these four
  pairs of paths, there must be at least one pair where neither path
  of the pair is a chord.  Otherwise, $C$ would have a cycle of four
  edges.  And because no two distinct $q$-simplices share two distinct
  $p$-faces, this is impossible in a submatrix of $\Bdmat$.

  Suppose without loss of generality that neither $a_4$ nor $a_5$ is a
  chord.  Then these two paths form a cycle.  If this cycle is not
  induced, that means it contains both end points of a potential chord
  not shown, but not the chord itself. Then we alter the two paths so
  that whenever we encounter an end point of a potential chord, not in
  either path, we always choose to cross the chord.  Each altered path
  must still have the same end points shown in $C_4$, or $C_4$ would
  not apply.

  This gives us a cycle $Y$ that is induced.  While $a_4$ or $a_5$ may
  contain a potential chord, it is still true that neither can be a
  chord.  Therefore $C \setminus Y$ is also induced.  Because $C$ is
  b-odd, either $Y$ or $C \setminus Y$ must be b-odd.  But this
  contradicts $M$ being minimal.  Therefore graph $C_4$ cannot occur.

  If we suppose that one or more of the end points of the chords is of
  degree more than 4, each of these points still must be of even
  degree.  If any added paths connect diagonally, then we have the
  same case as graph $C_3$.

  If no added paths connect diagonally, then divide the paths of the
  graph into four sets: those connecting the tops of two different
  chords, those connecting the bottoms of two different chords, and
  two sets connecting end points of the same chord. If any added path
  loops back to its starting vertex, ignore it for now.  Because all
  four end points must be of even degree, if any one of these sets
  contains an odd number of paths, then they all must.  This is
  equivalent to case $C_2$.  If all four sets contain an even number
  of paths, then this is equivalent to case $C_4$, and cannot occur.
  Note that any paths we have ignored that loop back to the vertex
  where they began do not affect the parity of these other four sets
  of paths.

  Also note that paths shown in these abstractions may cross each
  other or themselves in ways not shown, but this will not affect our
  results, as explained below.
  \begin{itemize}
  \item If two paths that both connect the same end points cross, then
    because we are only discussing the existence and parity of the
    number of paths between end points, our argument is unaffected.
  \item If a path connecting the two top points of the chords crosses
    a path connecting the two bottom points, then Case $C_3$ applies.
  \item Case $C_3$ also applies if a path connecting the ends of the
    same chord crosses a path connecting the ends of the other chord.
  \item If a vertical path crosses a horizontal path, or if a path
    that starts and ends at the same point crosses paths from two
    different sets, it is possible none of the graphs shown apply. But
    there still must be a path between the flags that does not contain
    both end points of any chord.
  \end{itemize}

  Also note that the flags may actually be placed at an end point of a
  chord.  However, because the flags may only be placed at
  $q$-vertices, and $C$ is bipartite, they cannot be placed at both
  ends of the same chord.  Therefore the placement of flags will also
  not affect our results, and we may abstract this placement as shown.

  As a final step to show our result that a full traversal exists,
  simply cut out a hole in either $a_2$ or $a_4$ in graph $C_1$ to cut
  out one of the flags.  This altered $C_1$ then represents the
  untraversed remainder of $C$ one must encounter, if there were no
  way to reach a flag without meeting both end points of a chord.
\end{proof}

\noindent We now introduce a generalization of nonorientable surfaces
to arbitrary dimensional chains. This concept allows us to describe
the structure of $q$-chains that the minimal violation submatrices of
$\Bdmat$ correspond to.

\begin{definition}
  \label{def-orient_rev_chain}
  An \emph{orientation-reversing $q$-chain} $Q$ in the simplicial
  complex $K$ is an ordered chain of $q$-simplices $\sigma_0,
  \sigma_1, \dots, \sigma_{k-1}$ where each $\sigma_i$ has common
  $p$-faces $\tau_i$ with $\sigma_{(i + 1) \bmod k}$ and $\tau_{(i -
    1) \bmod k}$ with $\sigma_{(i - 1) \bmod k}$, and the sum of the
  $2k$ entries of $\Bdmat$ indicating each $\tau_i$ is a face of
  $\sigma_i$ and $\sigma_{(i + 1) \bmod k}$ is $2 \bmod 4$.  Each
  $\tau_i$ is an \emph{interior} $p$-simplex of the chain.  We allow
  $q$-simplices or interior $p$-simplices of $Q$ to be repeated, as
  long as for any two instances of such a simplex, the simplices of
  the other dimension immediately before and after these two instances
  form a set of four distinct simplices.  Each $p$-face of any $\sigma
  \in Q$ that is not also the face of either $\sigma_{(i + 1)\bmod k}$
  or $\sigma_{(i - 1)\bmod k}$, for some $i$ indicating the order of
  an instance of $\sigma$ in $Q$, is an \emph{exterior} $p$-simplex of
  $Q$.
\end{definition} 

The restriction on repetition is equivalent to no entry of $\Bdmat$
being used twice in this chain. It is then immediate that each b-odd
circuit in the bipartite graph representation of $\Bdmat$ represents
an orientation-reversing chain, and vice versa.  There is an MNTU
submatrix (MNTUS) $M$ of $\Bdmat$ whose columns correspond to the
$q$-simplices of $Q$ if and only if this b-odd circuit is induced, and
does not properly contain another induced b-odd circuit.  If there is
such an MNTUS, we call the rows of $\Bdmat$ that intersect $M$, which
correspond to the interior $p$-simplices of the orientation-reversing
chain, \emph{interior} rows. We denote by $Q_M$ the columns of
$\Bdmat$ corresponding to the $q$-simplices in the
orientation-reversing chain, and also call the rows of $\Bdmat$ that
correspond to exterior $p$-simplices \emph{exterior} rows. These are
the rows of $\Bdmat$ that do not intersect $M$, but have nonzero
entries in $Q_M$.

Note that if there are repeated simplices in $Q$, the ordering given
is not unique, and two different orderings may differ by more than a
choice of a starting simplex $\sigma_0$.  The repeated simplices imply
that the bipartite graph representation is not a cycle, and a choice
of ordering the simplices in $Q$ corresponds to a choice of traversal
of its bipartite graph. We now define a submatrix that minimally
violates total unimodularity in a stricter sense.

\begin{definition}
  \label{def-CMNTUS}
  For a given matrix $A$, a \emph{columnwise minimally non
    totally-unimodular submatrix}, or CMNTUS, $M$ of $A$ is an MNTUS
  where no MNTU $M'$ that is also a submatrix of $A$ exists such that
  the set of columns of $M'$ is a subset of the set of columns of $M$.
  If there is such an $M'$, then $M'$ is \emph{columnwise contained}
  in $M$.
\end{definition}

\noindent We describe a useful property of a CMNTUS of $\Bdmat$, and
illustrate the distinction between a MNTUS and a CMNTUS on a
$2$-complex in Figure~\ref{fig-MNTUSnotCMNTUS}.
\begin{theorem}
  \label{thm-odd_ext_rows_CMNTUS}
  If $M$ is a CMNTUS of $\Bdmat$, then each exterior row for $M$ has
  an odd number of nonzero entries in $Q_M$.
\end{theorem}

\begin{proof}
  Let $i$ be an exterior row of an arbitrary MNTU submatrix $M$ of
  $\Bdmat$ with an even number of nonzero entries of $i$ in $Q_M$. Let
  $C$ be the chordless b-odd circuit that is the bipartite graph
  representation of $M$.  If $C$ is not a cycle, split nodes as
  necessary to represent $C$ as a cycle, or ``wheel''.  Add the
  bipartite graph edges of $i$ in $Q_M$, and think of these edges of
  $i$ as spokes of the wheel $C$.  Call any portion of $C$ in between
  consecutive spokes, along with these spokes, a ``slice'' of the
  wheel.  Each of these slices then is a cycle, and there are an even
  number of these slices in the entire wheel.

  If $C$ is not a cycle, then we have a choice of ordering the wheel
  as we split nodes to create it.  To show our result, we choose this
  ordering in such a manner that there is no chord for any slice of
  the wheel (when thinking of any slice as a separate cycle).  If we
  think of the spoke ends as being flags,
  Theorem~\ref{thm-induced_flag_trav} shows this step can be
  performed.

  If we start with a single slice of the wheel, and re-build the wheel
  by adding adjacent slices, it must be true that at least one of
  these slices must be a b-odd cycle. After putting together an odd
  number of b-even slices, the cycle that is the portion of the wheel
  we have built plus the two boundary spokes must be b-even. Hence if
  we put all but one of the slices of the wheel together, the
  resulting boundary is a b-even cycle. The last slice, unlike all the
  other previous slices, must have two edges in common with the part
  of the wheel already built.  We know the entire wheel is b-odd, and
  so this last slice must be b-odd.  We may use a similar b-parity
  argument to show that the number of b-odd slices is odd.

  Now restore $C$ to its original form.  The slice of the wheel that
  was a b-odd cycle is now a b-odd circuit.  And by our choice of
  traversal, it is chordless.  The slice contains two edges not in
  $C$, and because no two $q$-simplices may have more than one common
  $p$-face, excludes more than two edges of $C$.  Therefore the
  submatrix $M'$ whose nonzero entries are the edges of $C$ contains
  fewer columns than $M$.  Also, all the columns of $M'$ are also
  columns of $M$.  And because $C$ is a chordless b-odd circuit, $M'$
  is an MNTU. But this result contradicts the assumption that $M$ is a
  CMNTUS.
\end{proof}

\begin{figure}[ht!]
  \centering
  \includegraphics[scale=0.8]{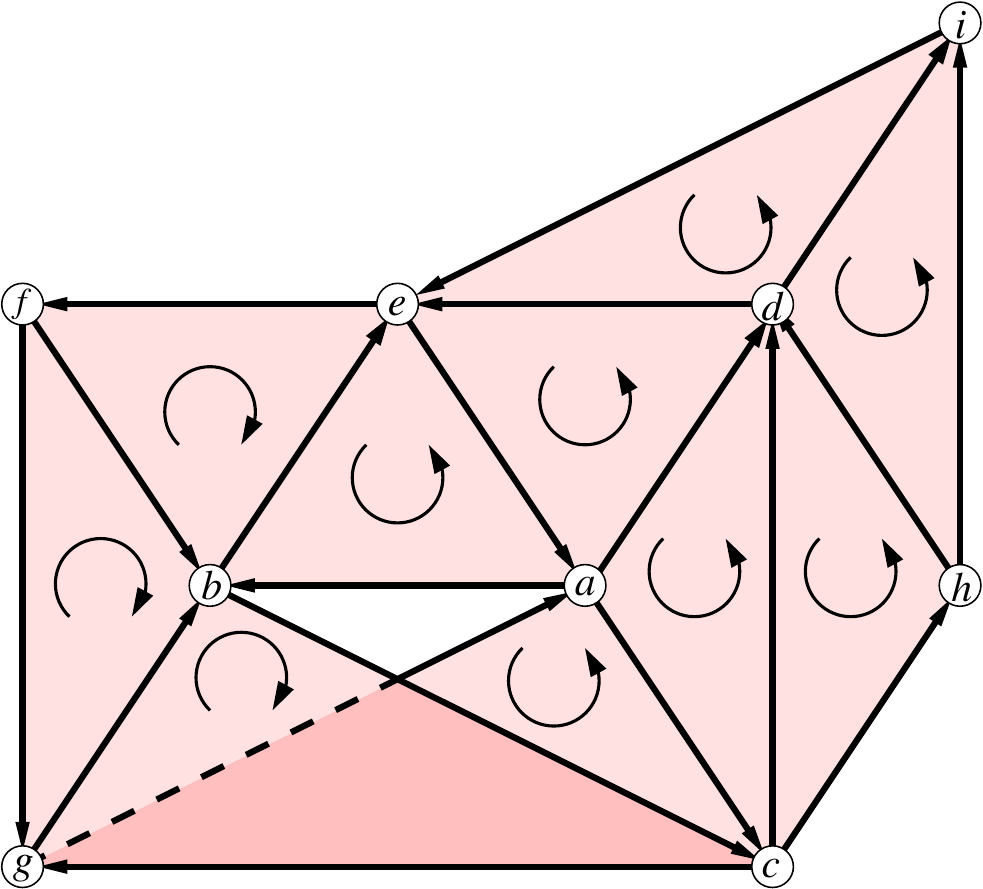}
  \caption{A $2$-complex illustrating an MNTUS and a CMNTUS. The
    M\"obius strip formed by all triangles represents an MNTUS $M$ of
    the $2$-boundary matrix of the complex, but $M$ is not a
    CMNTUS. Edge $ad$ is the face of two triangles in this M\"obius
    strip, and thus corresponds to an exterior row that has an even
    number of nonzeros in $Q_M$. The smaller M\"obius strip formed by
    the seven triangles leaving out $die, dhi$, and $dch$ represents a
    CMNTUS $M'$. Edge $ad$ is an interior row of $M'$.}
  \label{fig-MNTUSnotCMNTUS}
\end{figure} 

For any MNTUS $M$ of $\Bdmat$ with $r$ rows, let $\mathscr{M}$ be the
set of elements of $\Ker(A)$ whose nonzero $q$-coefficients are
contained in $Q_M$.  Because $\det(M)$ is nonzero, $\Ker(M)$ is
trivial.  This means there is a bijection between the set of linear
combinations of columns of $M$, and the set of possible row sums of
$M$.  This, along with Lemma~\ref{lem-vk_equiv_bi}, implies that for
any $\vm^1, \vm^2 \in \mathscr{M}$, the set of $p$-coefficients of
interior rows of $\vm^1$ and $\vm^2$ are equal if and only if all
$q$-coefficients of $\vm^1$ and $\vm^2$ are equal.

\begin{definition}
  For any row $i$ of $M$, let $[\vm^i]$ denote the equivalence class
  of elements of $\mathscr{M}$ whose $p$-coefficients of interior rows
  is the unit vector with its nonzero coefficient at row $i$.
\end{definition}
Letting $P_M$ represent the set of interior rows of $M$, another
consequence of the kernel of $M$ being trivial, and the bijection
between the set of linear combinations of columns of $M$ and the set
of possible row sums of $M$, is that the following equation holds for
any $\vm \in \mathscr{M}$ with $p$-coefficients $\vp$, and an
appropriate choice of $\vm^i$ from each $[\vm^i]$.
\begin{equation}
  \label{eqn-lin_cmb_unit_nulls}
  \vm = \sum_{i \in P_M}p_i\vm^i. 
\end{equation}

\noindent The following lemma characterizes the fractional elements of
$[\vm^i]$ for any interior row $i$ of $M$.
\begin{lemma}
  \label{lem-unit_nulls}
  For any MNTU submatrix $M$ of $\Bdmat$, and any interior row $i$ of
  $M$, each $q$-coefficient of any element of $[\vm^i]$ is nonzero if
  and only if it corresponds to a column of $Q_M$, and each such
  coefficient is $\pm (1/2)$.
\end{lemma}
\begin{proof}
  For an MNTUS $M$ of $\Bdmat$, let $i$ be an interior row.  Because
  $\det(M) \neq 0$, there is some entry $M_{ij}$ that is nonzero.  If
  we multiply $M_{ij}$ by $-1$, and call the result $M'$, then $M'$ is
  Eulerian, and the sum of its entries is $0 \bmod 4$.  Therefore $M'$
  is totally unimodular.  Hence there is some set $\mathscr{J}$ of
  columns of $M'$ that we may multiply by $-1$, and if we call the
  result $M_{-\mathscr{J}}'$, the sum of each row of
  $M_{-\mathscr{J}}'$ must be $0, 1$, or $-1$.  Because
  $M_{-\mathscr{J}}'$ is also Eulerian, each of these row sums must be
  $0$. If we now multiply $M_{ij}$ by $-1$ again and call the result
  $M_{-\mathscr{J}}$, the row sum of $i$ is $\pm 2$, and every other
  row sum is still 0.  If the row sum of $i$ is $-2$, multiply all
  columns of $M_{-\mathscr{J}}$ by $-1$. Now call this result, whether
  this final inversion is necessary or not, $M_i$.  The row sum for
  $i$ is 2, and all other row sums are 0, and $M_i$ is $M$ with some
  (perhaps empty) set of columns of $M$ scaled by $-1$. Therefore by
  Lemma~\ref{lem-vk_equiv_bi} and equation
  (\ref{eqn-lin_cmb_unit_nulls}), any element of $\mathscr{M}$ whose
  nonzero $q$-coefficients agree with the scalings of $M_i$ is twice
  some $\vm^i \in [\vm^i]$, and all these coefficients are either 1 or
  $-1$.  Since any elements of $\mathscr{M}$ are equal in all
  $q$-coefficients if and only if they are equivalent, any $\vm^i \in
  [\vm^i]$ must have all nonzero $q$-coefficients be $\pm (1/2)$.
\end{proof}

\noindent From this result, the following Lemma is almost immediate,
and its proof is omitted.
\begin{lemma}
  \label{lem-null_pairs}
  For a given MNTUS $M$, and any list of interior rows $i_1,
  i_2,... i_n$, and elements $\vm^{i_1}, \vm^{i_2},... \vm^{i_n}$ of
  $[\vm^{i_1}], [\vm^{i_2}],... [\vm^{i_n}]$, respectively:
  \begin{enumerate}
    \renewcommand{\theenumi}{{P\arabic{enumi}}}
    \renewcommand{\labelenumi}{{P\arabic{enumi}.}} 
  \item \label {prp-pair_coeffs_0-1} If $n = 2$, all $q$-coefficients
    of both $\vm^{i_1} + \vm^{i_2}$ and $\vm^{i_1} - \vm^{i_2}$ are in
    $\{0, \pm 1\}$.
  \item \label{prp-inv_sign_invb_mod_2} If $n = 2$, for any $j \in
    Q_M$, the $j\ord\ q$-coefficient of $\vm^{i_1} + \vm^{i_2}$ is
    zero if and only if the $j\ord\ q$-coefficient of $\vm^{i_1} -
    \vm^{i_2}$ is nonzero.
  \item \label{prp-even_sum_int} If $n$ is even, $\sum_{\alpha =
    1,...n}\vm^{i_\alpha}$ is integral.
  \item \label{prp-odd_sum_halves} If $n$ is odd, every nonzero
    $q$-coefficient of $\sum_{\alpha = 1,...n}\vm^{i_\alpha}$ is
    nonintegral, with each of these nonzero $q$-coefficients of the
    form $\frac{k}{2}$ with $k$ an odd integer
  \end{enumerate}
\end{lemma}
\noindent Note that nowhere in Lemma~\ref{lem-null_pairs} is it
required that $i_\alpha \neq i_\beta$ for any $\alpha, \beta, \in 1,
2,... n$.

\begin{definition}
  For any MNTUS $M$ of $\Bdmat$, and any concise $\vz$, let $\vm(\vz)$
  be the unique element of $\mathscr{M}$ where $\{\vz, \vm(\vz)\}$ is
  linearly concise, $z_j = 0$ implies the $j\ord$ entry of $\vm(\vz)$
  is $\leq 0$ $~\forall j \leq 2(m+n)$, and for each interior row $i$
  of $M$, the $i\ord\ p$-coefficient of $\vz$ and $\vm(\vz)$ are
  equal.
\end{definition}
Equation~\ref{eqn-lin_cmb_unit_nulls} then restricts $\vm(\vz)$ to a
single equivalence class.  The requirement to be linearly concise with
$\vz$ fixes coefficients of $\vm(\vz)$ corresponding to nonzero $p$-
and $q$-coefficients of $\vz$, and also requires $\vm(\vz)$ to be
concise.  Finally, the requirement on zero entries of $\vz$ fixes the
remaining entries of $\vm(\vz)$ by dictating, in the case where a $p$-
or $q$-coefficient is nonzero in $\vm(\vz)$ but zero in $\vz$, which
entry among each pair of corresponding opposite entries is nonzero,
thus making $\vm(\vz)$ unique.

\begin{theorem}
  \label{thm-elem_vert_M}
  For any MNTUS $M$ of $\Bdmat$, and any interior row $i$, there is a
  unique vertex $\vz^i$ of $P_i$ whose nonzero $q$-coefficients are
  contained in $Q_M$ and whose $p$-coefficients of interior rows are
  all 0.  This vertex is $\vz^I - \vm(\vz^I)$ where $\vz^I$ is the
  identity solution to OHCP$_i$.  Furthermore, if $M$ is a CMNTUS of
  $\Bdmat$, $\vz^i$ is the only nonintegral vertex of $P_i$ whose
  nonzero $q$-coefficients are contained in $Q_M$.
\end{theorem}

\begin{proof}
  Let $\vz^i = \vz^I - \vm(\vz^I)$.  Then $\vz^i$ is feasible, and all
  $p$-coefficients of interior rows of $\vz^i$ are zero. Because
  $\Ker(M)$ is trivial, this is the only concise feasible solution
  whose nonzero $q$-coefficients are contained in $Q_M$, with the
  $p$-coefficients of interior rows all zero.  To show $\vz^i$ is a
  basic solution, we try to decompose $-\vm(\vz^I)$ into $\vz^C +
  \vz^D$ satisfying the properties specified in
  Theorem~\ref{thm-v_plus_vk_basic}. Because $\Ker(M)$ is trivial,
  there must be some $x$-coordinate that is nonzero in $\vz^I +
  \vz^C$, but zero in $\vz^i$, implying this coordinate is also
  nonzero in $\vz^D$.  Therefore Property \ref{cnd-v_contains_vr} in
  Theorem~\ref{thm-v_plus_vk_basic} cannot be satisfied, and so
  $\vz^i$ is a basic solution of OHCP$_i$.

  Now assume $M$ is a CMNTUS.  Let $\vz$ be a basic solution not
  equivalent to $\vz^i$ whose nonzero $q$-coefficients are contained
  in $Q_M$.  Then by Lemma~\ref{lem-last_OHCP_term_cancels}, the
  $p$-coefficient for some exterior row $i_0$ is nonzero in $\vz^i$
  but zero in $\vz$.

  Since $M$ is a CMNTUS, by Theorem~\ref{thm-odd_ext_rows_CMNTUS},
  $i_0$ has an odd number of nonzero entries in $Q_M$.  Then
  Lemma~\ref{lem-unit_nulls} implies that an odd number of the
  $q$-coefficients for the columns with these nonzero entries must be
  integral in $\vz$.

  Decompose $\vz$ into $\vz^I + \vz^R + \vz^Z$, where $\vz^R$ is the
  element of $\mathscr{M}$ whose nonzero $q$-coefficients are equal to
  the nonintegral $q$-coefficients of $\vz$, and $\vz^Z$ is the
  element of $\mathscr{M}$ whose nonzero $q$-coefficients are equal to
  the integral $q$-coefficients of $\vz$. Then $\vz^R$ and $\vz^Z$
  cannot both be $\vzero$.  If $\vz^I + \vz^R$ is a basic solution,
  then by Cramer's Rule, this implies there is a non-TU matrix
  contained in the columns with nonzero $q$-coefficients in $\vz^R$,
  contradicting $M$ being a CMNTUS.

  If $\vz^I + \vz^R$ is not a basic solution, then by
  Theorem~\ref{thm-vk_active}, there is some nonzero $\vz^K \in
  \Ker(A)$ where all $p$- and $q$-coefficients nonzero in $\vz^K$ are
  also nonzero in $\vz^I + \vz^R$.  Since $\vz$ is a basic solution,
  $\vz^Z$ must cancel one of these coefficients. And since no nonzero
  $q$-coefficients in $\vz^R$ are nonzero in $\vz^Z$, these canceled
  coefficients must all be $p$-coefficients.  Since all
  $q$-coefficients of $\vz^Z$ are integral, all $p$-coefficients of
  $\vz^Z$ are also integral.  Then by Lemma~\ref{lem-must_cancel_two},
  there must be at least two nonzero integral $p$-coefficients in
  $\vz^I + \vz^R$, and so there must be at least one in $\vz^R$.  If
  we construct the input chain whose nonzero coefficients are the
  values of each of the integral $p$-coefficients of $\vz^R$
  multiplied by $-1$, then $\vz^R$ added to the identity solution for
  the OHCP with this input chain must be a basic solution to this
  OHCP.  Then by Lemma~\ref{lem-frac_basic}, there must be some
  OHCP$_j$ where $\vz^R$ added to the identity solution for OHCP$_j$
  is basic. Hence, again by Cramer's Rule, there is a non-TU submatrix
  contained in the columns with nonzero $q$-coefficients in $\vz^R$,
  contradicting $M$ being columnwise minimal.
\end{proof}

\begin{remark}
  \label{rem-zplusminusi}
  The result in Theorem~\ref{thm-elem_vert_M} holds for the case of
  $-\ve_i$ as the elementary input chain (for OHCP$_{-i}$), instead of
  the standard case of $\ve_i$. But in this case, the unique vertex,
  which we call $\vz^{-i}$, will be distinct from $\vz^i$ as used in
  the original statement of the Theorem. In further discussion, it is
  understood that when we refer to $\vz^i$, we cover both these
  possibilities.
\end{remark}

\begin{lemma}
  \label{lem-basic_elem_coeffs}
  For a set of columns $Q$ of $\Bdmat$, let there be no $i$ such that
  OHCP$_i$ has a nonintegral basic solution in $\left(P_A\right)_i$
  whose nonzero $q$-coefficients are contained in $Q$. Then for any
  basic solution $\vz^Y$ of any $OHCP_i$ in $\left(P_A\right)_i$ whose
  nonzero $q$-coefficients are contained in $Q$, all $p$-coefficients
  of $\vz^Y$ are in $\{0, \pm1\}$.
\end{lemma}
\begin{proof}
  If $\vz^Y$ is a basic solution in $\left(P_A\right)_i$ to OHCP$_i$
  for some $i$ whose nonzero $q$-coefficients are contained in $Q$, it
  must be integral.  Let $\vz^K = \vz^Y - \vz^I$, where $\vz^I$ is the
  identity solution.  If $\vz^Y$ has a coefficient $\alpha$ for the
  $x$-coordinate $i_0$ where $\abs{\alpha} > 1$, then by Lemma
  \ref{lem-scalar_mult_basic} and Theorem~\ref{thm-v_plus_vk_basic},
  if we let $\vz^{I_0}$ be the identity solution to the OHCP LP $O_0$
  whose input chain has all zeros except for coordinate $i_0$, which
  has coefficient $\pm 1$ that is opposite in sign to the coefficient
  of $i_0$ in $\vz^Y$, then $\vz^{I_0} + \displaystyle
  \frac{1}{\alpha}\vz^K$ is a basic solution to $O_0$ with the
  $p$-coefficient of $i$ being the nonintegral value $-(1/\alpha)$.
  However, any basic solution of $O_0$ where the nonzero
  $q$-coefficients are contained in $Q$ must still be integral, giving
  us a contradiction.
\end{proof}

\section{NTU Neutralized Complexes}
\label{sec-neut_complexes}

We define the concept of NTU neutralization, and present results
characterizing this condition.
\begin{definition}
  \label{def-neutralized}
  For any interior row $i$ of an MNTUS $M$ of $\Bdmat$, let $\vk^i$
  represent a concise integral element of $\Ker(A)$ whose sum of
  $p$-coefficients of interior rows is odd, $\left( \vk^i - \vm(\vk^i)
  \right)\toX \neq \vzero$, and let the absolute value of each
  $p$-coefficient of $\vk^i - \vm(\vk^i)$ be less than or equal to the
  absolute value of this coefficient in $\vz^i$.  If each interior row
  $i$ of $M$ has such a $\vk^i$, then $M$ is \emph{neutralized}.  If
  all MNTU submatrices of $\Bdmat$ are neutralized, then $K$ is
  \emph{NTU neutralized in the $q\ord$ dimension}.
\end{definition}
\begin{theorem}
  \label{thm-MNTUS_vert_toX}
  For any MNTUS $M$ of $\Bdmat$, the projection $\vz^i\toX$ for each
  interior row $i$ is a convex combination of $\vz^1\toX$ and
  $\vz^2\toX$ where both $\vz^1$ and $\vz^2$ are integral elements of
  $P_i$ if and only if $M$ is neutralized.
\end{theorem}
\begin{proof}
  First, we the ignore the restriction that $\vz^1$ and $\vz^2$ are
  integral, and say by Corollary~\ref{cor-vk_active_X} that each
  $\vz^i\toX$ is not a basic solution to OHCP$_i\toX$, and hence a
  convex combination of $\vz^1\toX$ and $\vz^2\toX$ with $\vz^1, \vz^2
  \in P_i$, if and only if there exists a $\vk_{\text{ex}}^i$ with
  $\vk_{\text{ex}}^i\toX \neq \vzero$, and where all nonzero
  coefficients of $\vk_{\text{ex}}^i\toX$ are nonzero in $\vz^i$.

  If for a given $\vz^i$ there is a $\vk^i$ satisfying
  Definition~\ref{def-neutralized}, by Lemma~\ref{lem-equiv_coll_conc}
  we may adjust $\vk^i$ and $\vm(\vk^i)$ if necessary so that
  $\{\vk^i, \vz^i, \vm(\vk^i)\}$ is linearly concise. Then $\vk^i -
  \vm(\vk^i)$ satisfies the conditions for $\vk_{\text{ex}}^i$.  So
  let $\vk_{\text{ex}}^i = \vk^i - \vm(\vk^i)$.  Since $\vk^i$ is
  integral, by Lemma~\ref{lem-null_pairs}, each $q$-coefficient in
  $Q_M$ of $\vk_{\text{ex}}^i$ is nonintegral with a denominator 2.
  By Lemma~\ref{lem-unit_nulls} and Theorem \ref{thm-elem_vert_M},
  this is also true of $\vz^i$.  Every other $q$-coefficient in both
  $\vk_{\text{ex}}^i$ and $\vz^i$ is integral. Hence if we let $\vz^1
  = \vz^i + \vk_{\text{ex}}^i$, and $\vz^2 = \vz^i -
  \vk_{\text{ex}}^i$, then both $\vz^1$ and $\vz^2$ are integral.
  Because the absolute value of each $p$-coefficient of
  $\vk_{\text{ex}}^i$ is less than or equal to the absolute of this
  coefficient in $\vz^i$, $\vz^1\toX$ and $\vz^2\toX$ are both in
  $P_i\toX$.  If either $\vz^1$ or $\vz^2$ is not in $P_i$, then by
  Corollary~\ref{cor-int_iff_int} we may use the same method as in
  Corollary~\ref{cor-equiv_vert} to transform either into some
  solution that is concise, integral, feasible, equivalent, and with
  all $x$-coefficients unchanged, keeping $\vz^i\toX$ a convex
  combination of $\vz^1\toX$ and $\vz^2\toX$.

  Now suppose that for some $i$, $\vz^i\toX$ is a convex combination
  of $\vz^1\toX$ and $\vz^2\toX$ with $\vz^1$ and $\vz^2$ feasible and
  integral.  Then there must be some $\vk_{\text{ex}}^i$ satisfying
  the same qualities as above, and some $\alpha >0$ and $\beta < 0$
  where $\vz^i + \alpha\vk_{\text{ex}}^i$ and $\vz^i +
  \beta\vk_{\text{ex}}^i$ are both integral with their projections in
  $P_i\toX$.

  Let $\alpha$ be of least absolute value such that $\vz^i +
  \alpha\vk_{\text{ex}}^i$ is integral. Because both $(\vz^i +
  \alpha\vk_{\text{ex}}^i)\toX$ and $(\vz^i -
  \alpha\vk_{\text{ex}}^i)\toX$ are in $P_i\toX$, the absolute value
  of each $p$-coefficient of $\alpha\vk_{\text{ex}}^i$ is less than
  the absolute value of this coefficient in $\vz^i$.

  All $q$-coefficients of $\alpha\vk_{\text{ex}}^i$ in $Q_M$ must be
  nonintegral with denominator 2, and all other $q$-coefficients
  integral.  Hence for any interior row $j$, there is some $\vm^j \in
  [\vm^j]$ where $\alpha\vk_{\text{ex}}^i + \vm^j$ is concise and
  integral.  Then any such $\alpha\vk_{\text{ex}}^i + \vm^j$ satisfies
  the conditions for $\vk^i$.
\end{proof}

\noindent We now present out main result, which states that the
complex being NTU neutralized is equivalent to none of the nonintegral
vertices of the OHCP LP projecting down to vertices in the projection
$P\toX$. Hence we cannot have a unique fractional optimal solution
when the complex is NTU neutralized.
\begin{theorem}
  \label{thm-K_neutralized}
  For a given complex $K$ with boundary matrix $\Bdmat$, the
  projection $\vz\toX$ of each nonintegral vertex $\vz$ of any OHCP LP
  over $K$ with integral input $p$-chain $\vc$ and polyhedron $P$ is
  not a vertex of $P\toX$ if and only if $K$ is NTU neutralized in the
  $q\ord$ dimension.
\end{theorem}
\begin{proof}
  First, we show that $K$ being neutralized is a necessary condition.
  If $K$ is not neutralized, then by Theorem~\ref{thm-MNTUS_vert_toX},
  for some $i$ and MNTU $M$, there is a $\vz^i$ where $\vz^i\toX$ is
  not a convex combination of $\vz^1\toX$ and $\vz^2\toX$ with both
  $\vz^1$ and $\vz^2$ integral elements of $P_i$.  If $\vz^i\toX$ is
  not a convex combination of any two points in $P_i\toX$, then it is
  a vertex of $P_i\toX$.

  If $\vz^i\toX$ is a convex combination of $\vz^1\toX$ and
  $\vz^2\toX$, then by Corollary~\ref{cor-vk_active_X}, and taking the
  notation from Theorem~\ref{thm-MNTUS_vert_toX}, there is some
  $\vk_{\text{ex}}^i$ and some rational number $\alpha > 0$ such that
  $\left(\vz^i \pm \alpha\vk_{\text{ex}}^i\right)\toX$ is in
  $P_i\toX$.  Since all variables of OHCP$_i$ are bounded, there
  exists a largest value of $\alpha$ for which this is true. Let
  $\alpha$ have this largest possible value.  Then either $\left(\vz^i
  + \alpha\vk_{\text{ex}}^i\right)\toX$ or $\left(\vz^i -
  \alpha\vk_{\text{ex}}^i\right)\toX$ brings one of the nonzero
  coefficients of $\vk_{\text{ex}}^i\toX$ to $0$.  Suppose without
  loss of generality that this condition is true for $\left(\vz^i +
  \alpha\vk_{\text{ex}}^i\right)\toX$.  If $\left(\vz^i +
  \alpha\vk_{\text{ex}}^i\right)\toX$ is not a basic solution of
  (OHCP$_i)\toX$, this means there is some other
  $\vk_{\text{ex}}^{i\prime}$, with some other $\alpha'$.  In this
  case, let $\alpha\vk_{\text{ex}}^i$ represent the sum of each such
  $\alpha\vk_{\text{ex}}^i$, making $\left(\vz^i +
  \alpha\vk_{\text{ex}}^i\right)\toX$ a basic solution.

  If $\vz^i + \alpha\vk_{\text{ex}}^i$ is not integral, then by
  Corollary~\ref{cor-equiv_vert_X}, and
  Lemma~\ref{lem-pre-im_vert_has_vert}, there is some corresponding
  vertex $\vz^1$ of OHCP$_i$ that is nonintegral, and $\vz^1\toX$ is a
  vertex of $P_i\toX$. If $\vz^i + \alpha\vk_{\text{ex}}^i$ is
  integral, let $\beta$ be a nonzero rational value opposite in sign
  to $\alpha$.  There must be some such $\beta$ where $\left(\vz^i +
  \beta\vk_{\text{ex}}^i\right)\toX$ is in $P_i\toX$.  If there is a
  maximum absolute value for such a $\beta$, then because $\vz^i\toX$
  is not a convex combination of $\vz^1\toX$ and $\vz^2\toX$, with
  both $\vz^1$ and $\vz^2$ integral, $\vz^i + \beta\vk_{\text{ex}}^i$
  is nonintegral, and by a similar logic as the $\alpha$ case, there
  is some $\vz^2$ of OHCP$_i$ that is nonintegral, and $\vz^2\toX$ is
  a vertex of $P_i\toX$.

  If $\vz^i + \alpha\vk_{\text{ex}}^i$ is integral, then if there is
  no upper bound for the absolute value of $\beta$, because $\Ker(A)$
  is rational, there must be some $\vz^i + \beta\vk_{\text{ex}}^i$
  that is integral, contradicting $\vz^i\toX$ not being a convex
  combination of $\vz^1\toX$ and $\vz^2\toX$, with both $\vz^1$ and
  $\vz^2$ integral.  Therefore $K$ being NTU neutralized in the
  $q\ord$ dimension is necessary for the projection $\vz\toX$ of each
  nonintegral vertex $\vz$ of any OHCP LP over $K$ with integral input
  $p$-chain $\vc$ and polyhedron $P$ to not be a vertex of $P\toX$.

  Now assume all MNTUS of $\Bdmat$ are neutralized.  Suppose for some
  OHCP LP, there is a nonintegral vertex $\vz$ where $\vz\toX$ is a
  vertex of $P\toX$.  Then by Theorem~\ref{thm_frac_vert_X}, for some
  $i$, there is a nonintegral $\vz'$ where $\vz'\toX$ is a vertex of
  $P_i\toX$, and Corollary~\ref{cor-frac_vert_X} implies that there is
  such a $\vz'$ where all nonzero $q$-coefficients are nonintegral.
  We will attempt to find a minimal set $Q$ of columns of $\Bdmat$
  that contains the nonzero $q$-coefficients of such a $\vz'$.

  By Cramer's Rule, and the definition of columnwise minimality, $Q$
  must contain the columns of a CMNTUS $M$ where $i$ is an interior
  row of $M$.  But we know from Theorem~\ref{thm-MNTUS_vert_toX} this
  is not enough because of some $\vk^i$ with $\vz' \pm \left(\vk^i -
  \vm(\vk^i)\right)$ integral.

  By Corollary~\ref{cor-must_cancel_two_X}, we must add $q$-simplices
  until there exist $p$-coefficients $r$ and $s$ as in
  Lemma~\ref{lem-must_cancel_two}. Hence $r$ and $s$ are exterior rows
  of $M$.  Further, $r$, $s$, and $i$ are interior rows of some
  orientation-reversing $q$-chain $C$.

  By Lemma~\ref{lem-basic_elem_coeffs}, all newly added
  $q$-coefficients, and therefore all $q$-coefficients of $C$, are the
  same in absolute value.  So if the columns of $C$ contain an MNTUS,
  the $q$-coefficients of the simplices in $C$ is a solution of the
  form $\vz^{i}$ (see Remark~\ref{rem-zplusminusi}).  Therefore
  Theorem~\ref{thm-MNTUS_vert_toX} applies and we have another $\vz^i
  \pm \left(\vk^i - \vm(\vk^i)\right)$ that is integral.  Any solution
  basic in OHCP$_i$ must be some combination of this $\vz^i$ and the
  previous $\vz'$, and so the projection of any such basic solution
  must be a convex combination of at least two projections of these
  four integral values.

  If the columns of $C$ do not contain an MNTUS, then by Cramer's
  Rule, no nonintegral vertices can be added.

  So we still have not found any vertex of $P_i$ whose projection is a
  vertex of $P_i\toX$.  But each time we add $q$-coefficients in a
  minimal way to find such a vertex, we may repeat the same logic as
  above.  Therefore no such vertex can exist.
\end{proof}

\begin{remark} \label{rem-exmpl}
  In either the left or right triangulation of the Example in Section
  \ref{ssec-exmpl} for the OHCP with input chain $ef$, $\vz^i\toX$ is
  all black solid edges, each with a coefficient of $0.5$ (see
  Figure~\ref{fig-threecplxs}).  $\vm(\vz^I)\toX$ is the union of
  these black solid edges, each with coefficient $-0.5$, together with
  edge $ef$ with coefficient $1$. In the right triangulation, the
  light and dark gray chains are projections of two integral vertices
  $\vz^1\toX$ and $\vz^2\toX$, respectively.  All coefficients in both
  of these projections are 1.  Triangle $adc$ satisfies all criteria
  for $\vk^i$.  Then $\vz^i\toX = (1/2)(\vz^1\toX + \vz^2\toX)$.
  Also, $\vm(\vk^i)\toX = \vk^i\toX - (1/2)(\vz^1\toX - \vz^2\toX)$,
  and $\vz^1\toX - (\vk^i - \vm(\vk^i))\toX = \vz^2\toX + (\vk^i -
  \vm(\vk^i))\toX = \vz^i\toX$.
\end{remark}

\subsection{Connections to integral polytopes and TDI systems} \label{ssec-conIPTDI}

There exist conditions weaker than the constraint matrix being TU,
which still guarantee that the LP has integral optimal solutions in
certain cases \cite[Chap.~21,22]{Schrijver1986}. In particular,
$k$-balanced matrices define a hierarchy of such matrices, with TU
matrices at one end \cite{CoCoVu2006}. The matrix $A$ is $k$-balanced
for any $k \in \Z_{>0}$ if $A_{ij} \in \{0,\pm 1\}$ and $A$ does not
contain an MNTUS with at most $2k$ nonzero entries in each row. If the
constraint matrix $A$ of the IP in Equation~(\ref{eqn-IP}) is
$k$-balanced, then for {\em certain} integral right-hand sides $\vb$,
the polytope of the associated LP is integral \cite{CoCoVu2006}. At
the same time, the polytope of the OHCP LP is not integral even when
the simplicial complex is NTU neutralized. Indeed, the constraint
matrix $A$ in not $k$-balanced for any $k$ in this case.

A linear system $A^T \vy \leq \vf$ is {\em totally dual integral}
(TDI) if the LP $~\min \, \{ \vf^T \vx ~|~ A \vx = \vc, \; \vx \ge
\vzero \}~$ has an integral optimal solution for every $\vc \in \Z^m$
for which the minimum is finite. Every OHCP instance has a finite
minimum, and when the complex is NTU neutralized, the OHCP LP is
guaranteed to have an integral optimal solution. Hence the linear
system $A^T \vy \leq \vf$ defined by the {\em dual} of the OHCP LP
(\ref{eqn-OHCP}) using $\vf^T = \begin{bmatrix} \vw^T & \vw^T &
  \vzero^T & \vzero^T \end{bmatrix}$ is TDI. This correspondence sheds
some light on the complexity of checking if a given complex is NTU
neutralized. The problem of checking if a linear system is TDI is
coNP-complete \cite{DiFeZa2008}, but could be done in polynomial time
if the dimension, or equivalently, $\rank(A)$ is fixed
\cite{CoLoSc1984,OSSe2012}.

\section{A Class of NTU Neutralized Complexes}
\label{sec-Sub-class}
In this section, we identify a class of complexes that may have
relative torsion, but are guaranteed to be NTU Neutralized.  This is
the class of $2$-complexes whose first homology group is trivial.

\begin{theorem}
  \label{thm-1_hom_triv_NTUNeut}
  If a $2$-complex has the trivial first homology group (over $\Z$),
  then it is NTU-neutralized.
\end{theorem}

\begin{proof}
  Any MNTU in $[\partial_2]$ is a M\"obius strip $M$ where each
  interior edge is the face of exactly two triangles of $M$.  First,
  let us assume each exterior edge is the face of exactly one triangle
  of $M$.  Then for any $\vz^i$, if we think of the elementary input
  edge as a path from its end points $v_1$ and $v_2$, this fractional
  vertex splits this path evenly into two distinct simple paths that
  together traverse all exterior edges of $M$.  Consider the direction
  of both these fractional paths as positive, so that each edge in
  $\vz^i$ has a coefficient of $+1/2$ under this consideration.

  Now consider the general case where exterior edges may be the face
  of more than one triangle in $M$.  The two fractional paths may not
  be simple now, and may loop back on themselves or each other.  Some
  of the resulting coefficients of exterior edges may now cancel each
  other out, and may even become negative.  But because these paths
  are simple and distinct in the previous case, with all coefficients
  of $+1/2$, the exterior edges still must contain a path from $v_1$
  to $v_2$ whose coefficients are all at least $+1/2$.  Call this path
  $h$.

  In the actual complex, the coefficients of $h$ in $\vz^i$ may be
  positive or negative. But they will have the same absolute value as
  our construction above.  Adding a path $h'$ of coefficient 1 in the
  opposite direction from $v_2$ to $v_1$ along $h$ is equivalent to
  adding $-1$ to each of the strictly positive coefficients.
  Therefore because each of the strictly positive coefficients were at
  least $+1/2$, adding this path from $v_2$ to $v_1$ will not increase
  the absolute value of any coefficients in $h$, or in fact for any
  coefficients in $\vz^i$.  The path $h'$, together with the
  elementary input edge, forms a loop $l$.

  Since the 1-homology of the complex is trivial, $l$ must be
  null-homologous, and so is equivalent to the $x$-coefficients of
  some integral element $\vk^i$ of $\Ker(A)$.  The sum of coefficients
  of interior edges in $\vk^i$ is odd, and $\vm(\vk^i) = \vm(\vz^I)$.
  Since $\vk^i\toX = \vz^I + h'$, and by the properties of $h'$
  relative to $\vz^i$ described above, the absolute value of each
  1-coefficient of $\vk^i - \vm(\vk^i)$ is less than or equal to the
  absolute value of this coefficient in $\vz^i$.  Therefore $\vz^i$ is
  neutralized, and since it is arbitrary, the complex is NTU
  Neutralized.
\end{proof}

\begin{remark}
  \label{rem-Add_two_triangs}
  In Figure~\ref{fig-threecplxs}, the left complex becomes a member of
  the subclass described in Theorem~\ref{thm-1_hom_triv_NTUNeut} if we
  add the triangles $akm$ and $dfh$.  As mentioned earlier, adding
  triangle $akm$ alone is enough to make the complex NTU neutralized,
  though.
\end{remark}

\begin{remark}
  \label{rem-checkH1rknotTU}
  For a $2$-complex, we could check more efficiently whether the first
  homology group is trivial \cite{DuHeSaWe2003} than checking whether
  the boundary matrix is TU \cite{Seymour1980}. If the former
  condition holds, then we can be sure that all OHCP instances on this
  complex can be solved efficiently.
\end{remark}

\section{Discussion} \label{sec-Disc}

Our results on MNTUS, in particular from Section
\ref{sec-BGs_of_MNTUs}, are specifically for such submatrices of the
boundary matrices of simplicial complexes. NTU neutralized complexes
define a class of LPs with unique structure -- these OHCP LP polytopes
may not be integral, yet, for every input chain, i.e., for every
integral right-hand side, there exists an integral optimal solution.
Our main result (Theorem \ref{thm-K_neutralized}) implies that when
$K$ is NTU neutralized, if an optimal solution of the OHCP LP is
nonintegral then there must exist another integral optimal solution
with the same total weight. If a standard LP algorithm finds the
fractional optimal solution, we should be able to find an adjacent
integral optimal solution using an approach similar to that of G\"uler
et al.~\cite{GuHeRoTeTs1993} for the same task in the context of
interior point methods for linear programming. This approach should
run in strongly polynomial time. 

While checking whether a linear system is TDI is coNP-complete, it is
not known whether a direct polynomial time approach could be devised
to check if the simplicial complex is NTU neutralized. Another
interesting question is whether the definition of the complex being
NTU neutralized could be simplified for low dimensional cases, which
could also be tested efficiently. We identify one class of simplicial
complexes that are guaranteed to be NTU neutralized
(Section~\ref{sec-Sub-class}). Are there other special classes of
complexes that are guaranteed to be NTU neutralized?  The NTU
neutralized complex (in right) in Figure \ref{fig-threecplxs}
illustrates a case where the same neutralizing chain neutralizes {\em
  all} relevant elementary chains. A characterization of the structure
of such complexes could also prove very useful.

\bibliographystyle{plain}

\end{document}